\documentclass[cmp]{svjour}

\usepackage[T1]{fontenc}

\usepackage[latin1]{inputenc}
\usepackage{graphicx,graphics}
\usepackage{amsthm}
\usepackage{amsfonts,amssymb,amscd,amsmath}
\usepackage[abs]{overpic}
\newtheorem*{mainthm}{Main Theorem}
\newtheorem*{maincor}{Main Corollary}
\newtheorem*{quest}{Question}

\newtheorem*{AGtheorem}{Quasiconformal Distortion of Angles Theorem}
\newtheorem*{Atheorem}{Entire Quasiconformal Maps are Quasisymmetric}

\newtheorem*{Ttheorem}{Teichm\"uller's Modulus Theorem}

\newcommand{\field}[1]{\mathbb{#1}}

\def\eps{\epsilon}

\def\fC{\field{C}}
\def\fD{\field{D}}

\def\fR{\field{R}}

\def\fZ{\field{Z}}

\def\cB{\mathcal{B}}

\def\cD{\mathcal{D}}

\def\cF{\mathcal{F}}

\def\cI{\mathcal{I}}

\def\cT{\mathcal{T}}

\def\cS{\mathcal{S}}

\def\qp{{P_\theta}}
\def\qt{{Q^t}}
\def\qth{{Q^{t(\theta)}}}
\def\rot{{R_\theta}}
\def\sd{{\Delta_\theta}}

\begin{document}
\title{On the scaling ratios for Siegel disks}
\author{Denis Gaidashev}
\institute{Department of Mathematics, Uppsala University, Uppsala, Sweden \\
{\tt gaidash@math.uu.se}
}
%

\setcounter{page}{1}

\maketitle


\begin{abstract}
The boundary of the Siegel disk of a quadratic polynomial  with an irrationally indifferent  fixed point and the rotation number whose continued fraction expansion is preperiodic has been observed to be self-similar with a certain scaling ratio. The restriction of the dynamics of the quadratic polynomial to the boundary of the Siegel disk is known to be quasisymmetrically conjugate to the rigid rotation with the same rotation number.  The geometry of this self-similarity is universal for a large class of holomorphic maps. A renormalization  explanation of this universality has been proposed in the literature. 

In this paper we provide an estimate on the quasisymmetric constant of the conjugacy, and use it to prove bounds on the scaling ratio $\lambda$ of the form 
$$\alpha^\gamma \le |\lambda| \le C \delta^s,$$
where $s$ is the period of the continued fraction,  and $\alpha \in (0,1)$ depends on the rotation number in an explicit way, while $C>1$, $\delta \in (0,1)$ and $\gamma \in (0,1)$ depend only on the maximum of the integers in the continued fraction expansion of the rotation number. 
\end{abstract}

\tableofcontents

\section{Introduction}
\label{intro}
\begin{definition}
Given $\theta \in (0,1]$, the quadratic polynomial $\qp$ is defined as 
$$\qp(z)=e^{2 i \pi \theta} z \left( 1-{z \over 2} \right).$$
\end{definition}

Here, the number  $\theta$ has a unique continued fraction expansion
$$\theta=[a_1, a_2, a_3, \ldots]= {1 \over a_1+{1 \over a_2 +{1 \over a_3 +\ldots } }}.$$
For $\theta \in (0,1]$, we denote, as usual, 
$${p_n \over q_n}=[a_1,a_2,a_3, \ldots, a_n]$$
the $n-th$ best rational approximant of $\theta$, obtained by truncating the continued fraction expansion for $\theta$.

In his classical work \cite{Si}, Siegel demonstrated that the polynomial $\qp$ is conformally conjugate to the linear rotation $\rot(z)=e^{2 i \pi \theta} z$ in a neighborhood of zero if $\theta$ is a Diophantine number. The maximal domain of this conjugacy is called the {\it Siegel disk}.  We will denote it as $\sd$. Bruno demonstrated in \cite{Bru1,Bru2} that an analytic germ with an irrationally neutral  multiplier $f'(0)=e^{2 i \pi \theta}$ is conformally linearizable in a neighborhood of zero if $\theta$ is a Bruno number, satisfying $\sum_{n=0}^\infty q_n^{-1} \ln q_{n+1}< \infty$.  Yoccoz proved in \cite{Yoc2} that the Bruno condition is also necessary for quadratic polynomials.

Numerical experiments (ex, \cite{MN} and \cite{Wid}) demonstrated that the boundaries of the Siegel disks of analytic germs $f$ with a multiplier $f'(0)=e^{2 i \pi \theta^*}$, $\theta^*={\sqrt{5}-1 \over 2}$, seemed to be non-differentiable Jordan curves which clearly exhibited a self-similar structure in the neighborhood of the {\it critical point} $c=1$. Furthermore, the limit $\lambda=\lim_{n \rightarrow \infty} (f^{q_{n+1}}(1)-f^{q_{n}}(1))/( f^{q_{n}}(1)-f^{q_{n-1}}(1))$ seemed to exist and to be {\it universal} - independent of the particular choice of $f$. This limit is called the {\it scaling ratio}.

The issue of nature of this curve was first addressed by Herman \cite{Her} and \'Swi{\c{a}}tek \cite{Sw2}. They proved that if $\theta$ is of {\it bounded type}, that is 
$$\sup{a_i} < \infty,$$ 
then $\partial \sd$ is a quasicircle (an image of the circle $|z|=1$ under a quisi-conformal map) containing the critical point $1$. Petersen \cite{Pet}  proved that the Julia set $J(\qp)$ is locally connected and has Lebesgue measure zero.

Some numerical observations of \cite{MN} and \cite{Wid} about the self-similarity of $\partial \sd$ have been proved analytically by McMullen in \cite{McM2}. In particular he proved that if $\theta$ has a preperiodic continued fraction expansion, that is a continued fraction expansion which is eventually periodic, then the small-scale dynamics of $\qp$ admits an asymptotic: certain high order iterates of $\qp$, appropriately rescaled, converge. Among other things, using renormalization for commuting holomorphic pairs, McMullen showed that if $\theta$ is a quadratic irrational and $s$ is the period of its continued fraction, then the limit 
\begin{equation}\label{eq:scaling_ratio}
\lambda=\left\{  
\lim_{n \rightarrow \infty} { \ f^{q_{n+s}}(1)-f^{q_{n}}(1) \ \over f^{q_{n}}(1)-f^{q_{n-s}}(1)}, \ s -{\rm even}, \atop
\lim_{n \rightarrow \infty} {  \ f^{q_{n+s}}(1)-f^{q_{n}}(1)  \ \over \overline{f^{q_{n}}(1)-f^{q_{n-s}}(1)} }, \ s -{\rm odd},
 \right. 
\end{equation}
indeed exists in the case $f \equiv \qp$, and the boundary of the Siegel disk is self-similar around the critical point.

\begin{figure}
\centering
\vspace{2mm}       
\begin{tabular}{c c} 
\includegraphics[height=5cm]{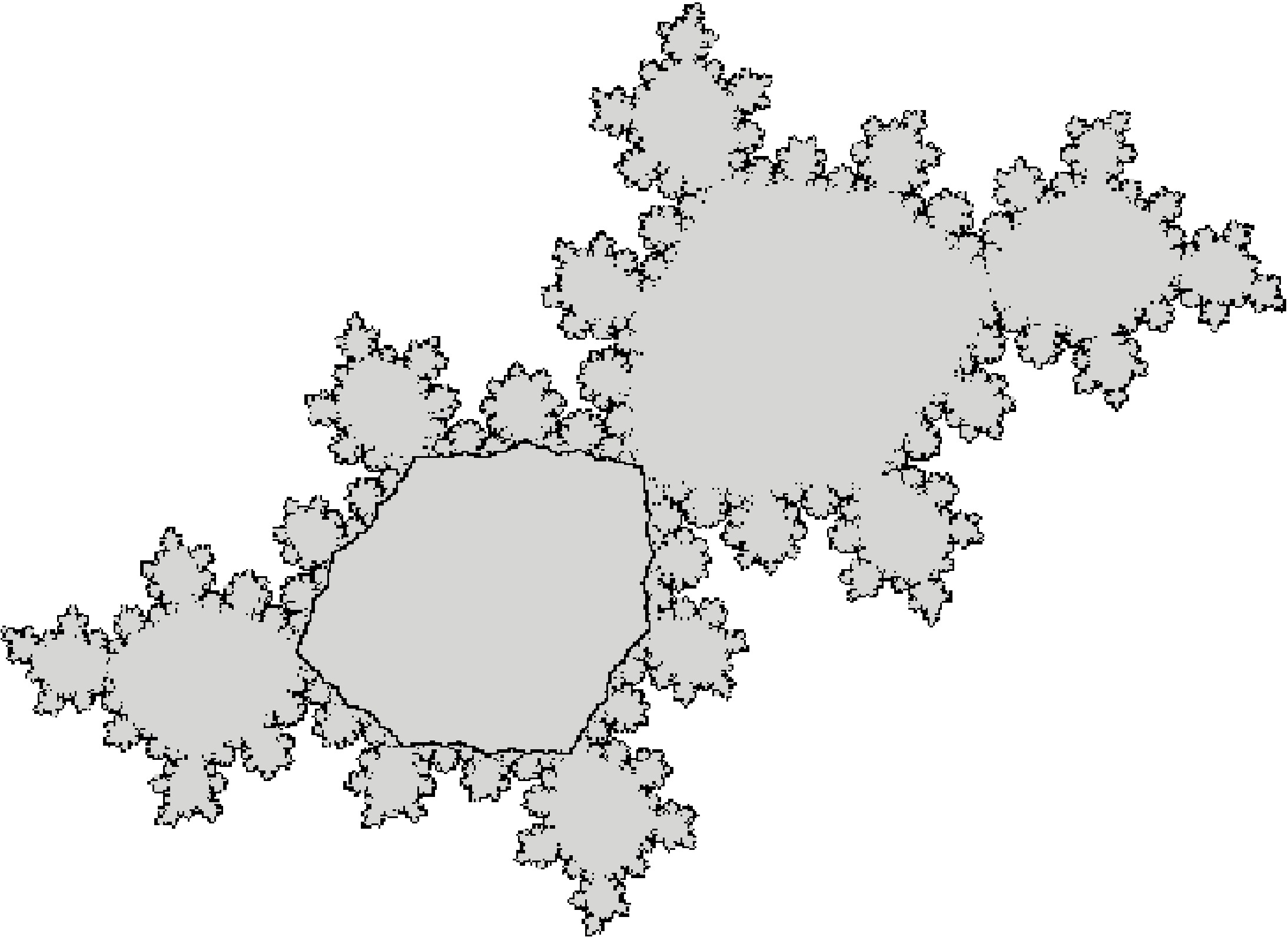} \quad & \quad \includegraphics[width=5cm,height=5cm]{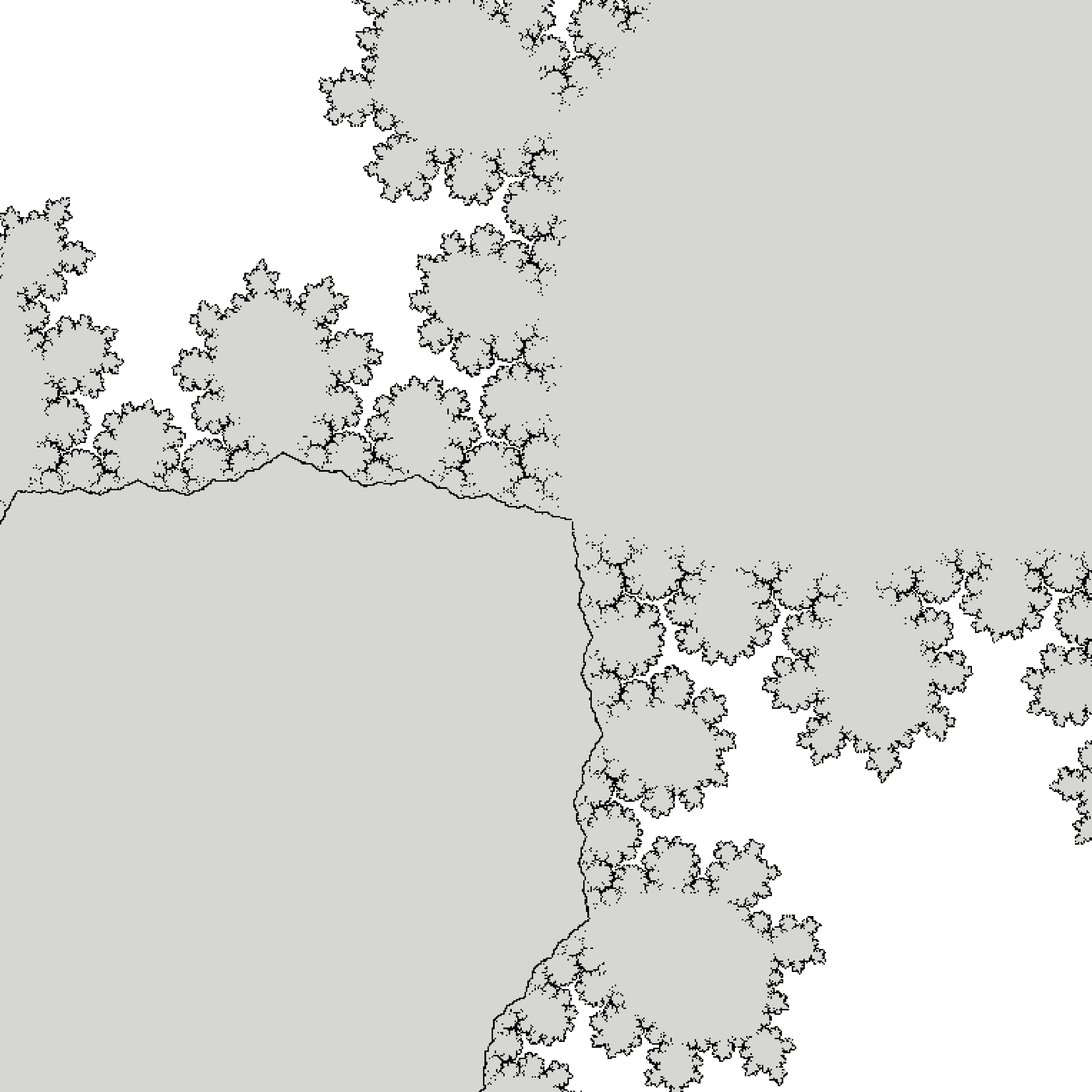}
\end{tabular}
\vspace{4mm}       
\caption{The filled Julia set of the golden mean quadratic polynomial $e^{2 i \pi \theta} z+ z^2$ (grey), together with a blow up around the critical point. The Siegel disk is bounded by the black curve.}
\label{fig:1}       
\end{figure}

Universal nature of the self-similarity can be explained if one could prove hyperbolicity of renormalization in an appropriate functional class. Hyperbolicity of renormalization for golden mean germs has been addressed in \cite{GaYa}, where the {\it cylinder renormalization} operator of \cite{Ya1} is used to demonstrate the existence of an unstable manifold in an appropriate Banach manifold of golden  mean germs; furthermore, the authors give an outline of a computer-assisted-prove of existence of a stable renormalization manifold, based on rigorous computer-assisted uniformization  techniques of \cite{Gai}. Hyperbolicty of renormalization is accessible analytically for golden-like rotation numbers $\theta=[N,N,N, \ldots]$, where $N$ is large; specifically, \cite{Ya2} uses the results of \cite{ISh} to prove the hyperbolicity conjecture for ``close-to-parabolic'' rotation numbers which contain a subsequence $a_{i(k)}$ of large integers in the continued fraction expansion.

\begin{figure}
\centering
{\includegraphics[width=6cm,height=6cm,angle=-90]{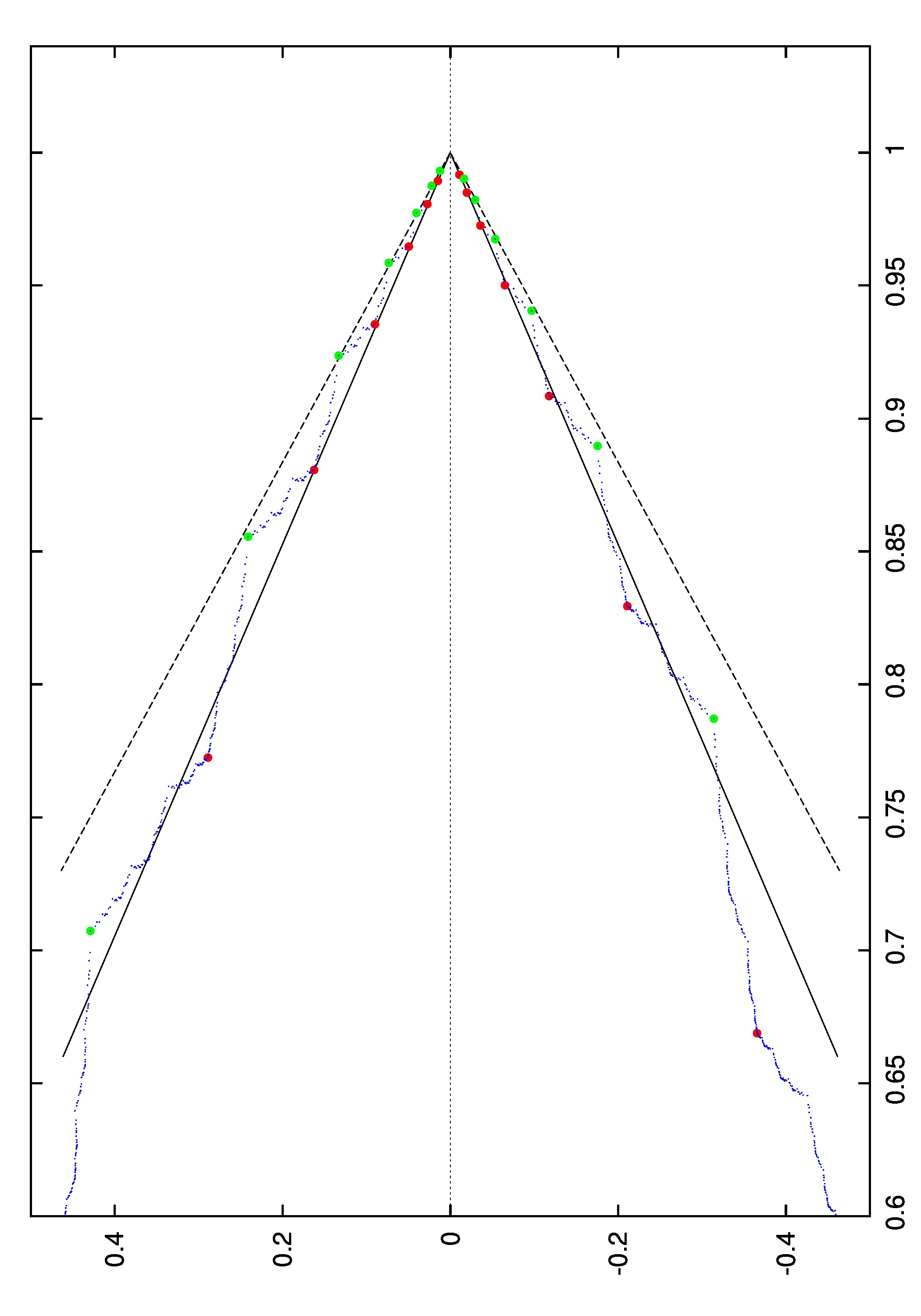}}
\vspace{3mm}       
\caption{The geometry of the forward (red) and backward (green) closest returns in the golden mean Siegel disk for the polynomial $P_\theta$. The two sets of lines are separated by $107.2\ldots^\circ$ and $119.6\ldots^\circ$ angles, respectively.}
\label{fig:2}       
\end{figure}

L. Carleson, based on the earlier numerical observations, has conjectured in \cite{Car} that the closest backward returns to the critical point of the golden mean Siegel disk converge on two lines separated by an angle equal to ${2 \over 3} \pi$, and suggested an approach to prove this.  The numerical results of \cite{MN} indicate, however, that close to the critical point, $\partial \sd$ is asymptotically contained between two sectors of angles $107.2\ldots^\circ$ and $119.6\ldots^\circ$ (see Fig. \ref{fig:2}), the second value being definitely less than ${2 \over 3} \pi$.

Several specific questions of self-similarity of Siegel disks have been addressed in \cite{BuHe}. We will now give brief summary of the results therein. 

Recall that a number $\theta \in (0,1]$ is a quadratic irrational iff its continued fraction is preperiodic: there exists an integer $N \ge 1$ such that $a_{i+s}=a_i$ for all $i \ge N$ and  for some integer $s \ge 1$. Given such $\theta$, denote
\begin{equation}\label{quad_ir}
\theta_i=[a_i,a_{i+1},a_{i+2}, \ldots], \quad  \alpha=\theta_{N+1} \theta_{N+2} \ldots \theta_{N+s}.
\end{equation}

\begin{theorem} (X. Buff, C. Henriksen)  \label{BHtheorem1}
Let $\theta$ be a quadratic irrational. Then the scaling ratio $\lambda  \in  \fD \setminus \{ 0\}$, 
$$\lambda=\left\{ 
\lim_{n \rightarrow \infty} { \  \qp^{q_{n+s}}(1)-\qp^{q_{n}}(1))  \ \over   \qp^{q_n}(1)-\qp^{q_{n-s}}(1)}, \ s - {\rm even}, \atop
\lim_{n \rightarrow \infty} { \ \qp^{q_{n+s}}(1)-\qp^{q_{n}}(1))  \  \over   \overline{\qp^{q_n}(1)-\qp^{q_{n-s}}(1)}}, \ s - {\rm odd},
 \right.
$$
of the Siegel disk $\sd$ about the critical point satisfies
$$\alpha < |\lambda| <1,$$ 
where $\alpha$ is as in $(\ref{quad_ir})$.
\end{theorem}

The second result of \cite{BuHe} goes in the direction of Carleson's conjecture.

\begin{theorem} (X. Buff, C. Henriksen)  \label{BHtheorem2}
If $-\pi/\ln(\alpha^2)>1/2$ then $\sd$ contains a triangle with a vertex at the critical point $c$. In particular, $\Delta_{\theta^*}$, where $\theta^*=(\sqrt{5}-1)/2$ is the golden mean, contains a triangle with a vertex at $c$.
\end{theorem}

The hypothesis of Theorem  $\ref{BHtheorem2}$ does not hold, for example, for $\theta=[N,N,N, \ldots]$ with $N \ge 24$. For such rotation numbers the method of \cite{BuHe} does not allow to prove existence of an inscribed  triangle with a vertex at the critical point.

X. Buff and C. Henriksen also put forward several questions, one of them being:


\medskip

\begin{quest}
 Are there constants $\delta_1<\delta_2<1$, such that $\delta_1^s \le |\lambda| \le  \delta_2^s$ where $s$ is the period of the quadratic irrational $\theta$?
\end{quest}

\medskip

In this paper we answer this question partially: we find  constants $\delta_1$ and $\delta_2$ which depend only on the maximum of integers in the continued fraction expansion of $\theta$. Specifically, we prove the following result.

\begin{mainthm}
Suppose that $\theta=[a_1,a_2,\ldots]$ is a quadratic irrational whose continued fraction is preperiodic with period $s$, and let $\{p_n / q_n\}$ be the sequence of the best rational approximants of $\theta$. Then there exists $C(n)>1$, with $\lim_{n \mapsto \infty} C(n)=1,$ and constants $A<1$, $\beta>1$ and $0<\gamma < 1$, dependent only on $\max\{a_i\}$, such that the following holds.

\medskip

\noindent 1) If $s$ is odd, then
$$\alpha^{\gamma}  \le {|\qp^{q_{n+s}}(1)-1| \over |\qp^{q_{n}}(1)-1|} \le {C(n) \over \sqrt{A \beta^{-\log_2{\alpha^2 \over 1-\alpha^2}+1}+1}}, \quad {\rm if} \quad  \alpha  > {1 \over \sqrt{2}},$$ 
and
$$\alpha^{\gamma} \le {|\qp^{q_{n+s}}(1)-1| \over |\qp^{q_{n}}(1)-1|}  \le C(n) \beta^{-{1 \over 2} [s \log_2 \vartheta]}, \quad {\rm if}  \quad  \alpha  \le  {1 \over \sqrt{2}},$$
where $[ \cdot ]$ denotes the integer part, while
$$\vartheta=\alpha^{-{1 \over s}}.$$

\medskip
\noindent 2) If $s$ is even, then
$$\alpha^{\gamma}  \le {|\qp^{q_{n+s}}(1)-1| \over |\qp^{q_{n}}(1)-1|} \le {C(n) \over A \beta^{-\log_2{\alpha \over 1-\alpha}+1}+1}, \quad {\rm if} \quad  \alpha  > {1 \over 2},$$ 
and
$$\alpha^{\gamma} \le {|\qp^{q_{n+s}}(1)-1| \over |\qp^{q_{n}}(1)-1|}  \le C(n) \beta^{-[s \log_2 \vartheta]}, \quad {\rm if} \quad  \alpha  \le {1 \over 2}.$$
\end{mainthm}

\begin{remark}
Notice, that the bounds of the Main Theorem are given for the quantity $|\qp^{q_{n+s}}(1)-1|/|\qp^{q_{n}}(1)-1|$ which is not exactly that in our definition $(\ref{eq:scaling_ratio})$ of the scaling ratio $\lambda$. However, by the result of McMullen from \cite{McM2}, the limit $\hat{\lambda}=\lim_{n \rightarrow \infty} (\qp^{q_{n+s}}(1)-1)/(\qp^{q_{n}}(1)-1)$ exists, and is different from $1$; then, a straightforward computation show that 
$${\qp^{q_{n+s}}(1)-\qp^{q_{n}} \over \qp^{q_{n}}(1)- \qp^{q_{n-s}}}-{\qp^{q_{n+s}}(1)-1 \over \qp^{q_{n}}(1)-1 }={
{\qp^{q_{n-s}}(1)-1 \over \qp^{q_{n}}(1)-1 } {\qp^{q_{n+s}}(1)-1 \over \qp^{q_{n}}(1)-1 } -1 \over 1-{\qp^{q_{n-s}}(1)-1 \over \qp^{q_{n}}(1)-1 }} \quad {\rightarrow \atop n \rightarrow \infty} \quad   {\hat{\lambda}^{-1} \hat{\lambda}-1 \over 1-\hat{\lambda}^{-1}}=0.$$
Thus, $\hat{\lambda}=\lambda$.
\end{remark}

\begin{maincor}
Suppose that $\theta=[a_1,a_2,\ldots]$ is a quadratic irrational whose continued fraction is preperiodic with period $s$, and let $\{p_n / q_n\}$ be the sequence of the best rational approximants of $\theta$. Then there exists constants $\delta_2<1$ and $C_2>1$ which depend on  $\max\{a_i\}$ only, such that,
$$\lim_{n \rightarrow \infty} {|\qp^{q_{n+s}}(1)-1| \over |\qp^{q_{n}}(1)-1|} \le  C_2 \delta_2^s.$$ 
\end{maincor}

\begin{proof}
Notice, that 
$$\vartheta=\alpha^{-{1 \over s}}=\left(\theta_{N+1} \theta_{N+2} \ldots \theta_{N+s} \right)^{-{1 \over s}} \ge \max_{1 \le  i \le s} \{\theta_{N+i}   \}^{-1} \ge  {1 \over [1,B,1,B,1,B, \ldots] } \equiv {\vartheta}_B,$$
where $B=\max_{ i \ge 1 } a_i.$
Now, the scaling ratios can be bounded from above by $s$-th powers of constants depending only on $B$. To demonstrate that we consider several cases separately.

\medskip

\noindent {\it $1)$ Case of odd $s$, $\alpha > 1/\sqrt{2}$.}
\begin{eqnarray}
\nonumber {|\qp^{q_{n+s}}(1)-1| \over |\qp^{q_{n}}(1)-1|} &\le& {C(n) \over \sqrt{ A \beta^{-\log_2{\alpha^2 \over 1-\alpha^2}+1}+1}} \le {C(n) \over A^{1 \over 2} \beta^{1 \over 2} } \left( {1 \over \beta} \right)^{{1 \over 2} \log_2 (\vartheta^{2 s}-1)} \\
\nonumber & \le& {C(n) \over A^{1 \over 2}  \beta^{1 \over 2}} \left( {1 \over \beta} \right)^{s {1 \over 2} \log_2 (\vartheta_B^{2 s}-1)^{1 \over s}} \le {C(n) \over A^{1\over 2}  \beta^{1 \over 2} }\left( {1 \over \beta^{{1 \over 2} \log_2 (\vartheta_B^2-1)} } \right)^s \equiv {C(n) \over A^{1 \over 2}  \beta^{1 \over 2}} \delta_2^s, \quad  \delta_2 \equiv {1 \over \beta^{{1 \over 2} \log_2 (\vartheta_B^2-1)} },
\end{eqnarray}
where we have used that for all $\vartheta_B$ (which is always larger than $1$) the function  $(\vartheta_B^{2 s}-1)^{1 \over s}$ is an increasing function of $s$. Therefore, 
$$ \lim_{n \rightarrow \infty} {|\qp^{q_{n+s}}(1)-1| \over |\qp^{q_{n}}(1)-1|} \le {1 \over A^{1 \over 2}  \beta^{1 \over 2}} \delta_2^s \equiv C_2 \delta_2^s.$$

\medskip

\noindent {\it $2)$ Case of even $s$, $\alpha > 1/2$.} 
\begin{eqnarray}
\nonumber {|\qp^{q_{n+s}}(1)-1| \over |\qp^{q_{n}}(1)-1|} &\le& {C(n) \over A \beta^{-\log_2{\alpha \over 1-\alpha}+1}+1} \le {C(n) \over A \beta} \left( {1 \over \beta} \right)^{\log_2 (\vartheta^s-1)} \\
\nonumber & \le& {C(n) \over A \beta} \left( {1 \over \beta} \right)^{s \log_2 (\vartheta_B^s-1)^{1 \over s}} \le {C(n) \over A \beta} \left( {1 \over \beta^{\log_2 (\vartheta_B-1)} } \right)^s \equiv {C(n) \over A \beta} \delta_2^s, \quad  \delta_2 \equiv  {1 \over \beta^{\log_2 (\vartheta_B-1)} }.
\end{eqnarray}

\medskip

\noindent {\it $3)$ Case $\alpha \le 1/\sqrt{2}$, $s$ odd.}
$$
{|\qp^{q_{n+s}}(1)-1| \over |\qp^{q_{n}}(1)-1|}  \le C(n) {1  \over \beta^{{1 \over 2} [ s \log_2 \vartheta_B ]}} \le C(n) {1  \over \beta^{{1 \over 2} s [\log_2 \vartheta_B ]}}  \equiv C(n) \delta_2^s, \quad \delta_2 \equiv  {1  \over \beta^{{1 \over 2} [\log_2 \vartheta_B ]}}.$$

\medskip

\noindent {\it $4)$ Case $\alpha \le 1/2$, $s$ even.}
$${|\qp^{q_{n+s}}(1)-1| \over |\qp^{q_{n}}(1)-1|}  \le C(n) {1  \over \beta^{ [ s \log_2 \vartheta_B ]}} \le C(n) {1  \over \beta^{ s [\log_2 \vartheta_B ]}}  \equiv C(n) \delta_2^s, \quad \delta_2 \equiv  {1  \over \beta^{[\log_2 \vartheta_B ]}}.$$

Thus, in all cases,  the upper bound depends solely on the period and the maximum of the integers in the continued fraction expansion. 
\end{proof}
\qed

In Section \ref{statements} we will give bounds on all constants in the above theorem. First, however, we will give a brief outline of the theory involved in the proofs, and quote several results from the literature that we will require.

\section{Preliminaries}

We will now give an introduction to four particular themes which will play an important role in our proofs: self-similarity of Siegel disks, quasiconformal conjugacy of the dynamics of the quadratic polynomial to that of a specific map called the modified Blaschke product, the distortion of angles and eccentricities under quasiconformal maps, and moduli of quadrilaterals.

\subsection{Self-similarity of the Siegel disk}
\label{slefsim}

Below, the rotation number $\theta$ will always be a quadratic irrational (recall $(\ref{quad_ir})$). McMullen demonstrates in \cite{McM2} that for $n \ge N$
$$\{ q_{n+s} \theta \}=(-1)^s \alpha \{ q_n \theta \}.$$

In a neighborhood of $1$, the map 
$$z \mapsto \left\{z^\alpha, \quad s \quad {\rm is \quad even},  \atop \bar{z}^\alpha, \quad s \quad {\rm is \quad odd},   \right.$$
conjugates $R_\theta^{q_n}$ to $R_\theta^{q_{n+s}}$ for all $n \ge N$. Furthermore, McMullen introduces in \cite{McM2} the mapping
$$\psi(z)=\left\{\phi^{-1}(\phi(z)^\alpha), \quad s \quad {\rm is \quad even},  \atop \phi^{-1}(\overline{\phi(z)}^\alpha), \quad s \quad {\rm is \quad odd},   \right.$$
where $\phi$ is the conformal isomorphism of the unit disk with $\sd$, normalized such that $\phi(1)=1$, and proves  that there exists a neighborhood $U$ of $1$, and a constant $\epsilon$ such that  $\psi$ is well defined on $U \cap \overline{\sd}$, conjugates $\qp^{q_n}$ to $\qp^{q_{n+s}}$, and is $C^{1+\epsilon}$-conformal or anticonformal:
$$\psi(z)=\left\{1+\lambda (z-1) + O(|z-1|^{1+\epsilon}), \quad s \quad {\rm is \quad even},  \atop 1+\lambda \overline{(z-1)} + O(|z-1|^{1+\epsilon}), \quad s \quad {\rm is \quad odd},   \right.$$
here, $\lambda=\psi'(1)$ is the scaling ratio.

\begin{figure}[t]
\centering
\begin{overpic}[scale=1,unit=1mm]{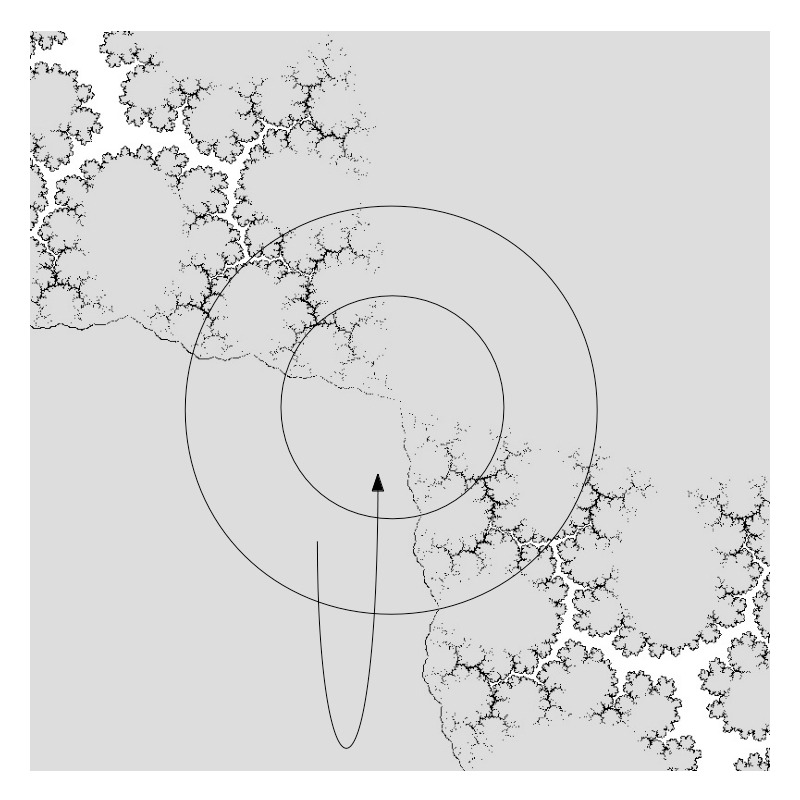}
\put(18,10){$\psi \circ \psi$}
\end{overpic}
\caption{Self-similarity of the Siegel disk. The filled Julia set in the interior of the smaller disk is a $C^{1+\eps}$-conformal copy of the filled Julia set in the interior of the larger disk.}
\label{fig:3} 
\end{figure}

The linearization of $\psi$ at $1$ will be called $\Lambda$:
$$\Lambda(z)=\left\{1+\lambda (z-1), \quad s \quad {\rm is \quad even},  \atop 1+\lambda \overline{(z-1)}, \quad s \quad {\rm is \quad odd}.   \right.$$

Fig. $\ref{fig:3}$ illustrates the self-similarity of a Siegel disk: the second iterate of $\psi$, which is $C^{1+\eps}$-conformal, maps the neighborhood inside the larger disk to that inside of the smaller. Given an open disk $D_\delta(1)$ of radius $\delta$ around $1$, the filed Julia set in its interior, $K_\qp \cap D_\delta(1)$, becomes an affine copy of $K_\qp \cap D_{|\lambda|^{-2} \delta}(1)$ under the map $\Lambda^2$ as $\delta \rightarrow 0$.

Since $\qp$ is a quadratic polynomial, $\sd$ has one preimage $\sd'$, different from and symmetric to $\sd$ with respect to $1$. McMullen proves in \cite{McM2} that the blow-ups $\Lambda^{-n}(\sd)$ and  $\Lambda^{-n}(\sd')$ converge in the Hausdorff topology on compact subsets of the sphere, to $\Lambda$-invariant quasidisks $\cD$ and $\cD'$, respectively. The boundaries of both of these quasidisks pass through $1$ and $\infty$.

Consider the cylinders $C=\cD / \Lambda^2$ and $C'=\cD' / \Lambda^2$ (we consider $\Lambda^2$ instead of $\Lambda$, because if $s$ is odd, then $\Lambda$ is orientation reversing). These two cylinders are conformally equivalent. Buff and Henriksen prove the following Lemma in \cite{BuHe} about the modulus of these cylinders, which plays an important role in their proof of Theorem $\ref{BHtheorem1}$. 

\begin{lemma} (X. Buff, C. Henriksen) \label{BHlemma}
The modulus of the cylinder  $C=\cD / \Lambda^2$ is equal to $-{ \pi / \ln{\alpha^2}}.$
\end{lemma}

We will also use this Lemma in our proof of the lower bound on the scaling ratio in Section $\ref{lower_bound}$.

\subsection{The Blaschke model of the Siegel disk}
\label{blaschke}

We will now present the construction of a Blaschke product model for a quadratic polynomial of Douady, Ghys, Herman and Shishikura. Our description will generally follow that of \cite{YaZa}. The reader is referred to this work for a more detailed description of the Blaschke model for the filled Julia set of a quadratic polynomial.

 Define 
$$\qt(z)=e^{2 \pi i t } z^2 \left( {z-3 \over 1-3 z} \right).$$

The restriction $\qt \arrowvert_{\field{T}}$, $\field{T}$ being the unit circle $\{z \in \field{C}: |z|=1\}$, is a critical circle map with the cubic critical point at $z=1$, and the critical value $e^{2 \pi t}$. By monotonicity, for each irrational $0 < \theta < 1$ there exists $t(\theta)$, such that the rotation number $\rho\left(Q^{t(\theta)} \arrowvert_{\field{T}} \right)=\theta$. Recall, that by the result of Yoccoz \cite{Yoc1}, any critical circle map with an irrational rotation number is topologically conjugate to the rotation $R_\theta$.

$\qth$ has a superattracting fixed points at $0$ and $\infty$ and a double critical point at $1$. $\qth$ acts as a double branched covering of the immediate basin of attraction of $\infty$, $\cB(\infty)$. Since $\qth$ commutes with the reflection $T(z)=\bar{z}^{-1}$, it also acts as a degree $2$ covering on the immediate basin of attraction of the origin, $\cB(0)$. 

As usual, let $\field{T}$ denote the unite circle $\{z \in \fC: |z|=1 \}$. Fix an irrational rotation number $0 < \theta < 1$ of bounded type. By the theorem of Herman and \'Swi\c{a}tek, the unique homeomorphism $h: \field{T} \mapsto \field{T}$ with $h(1)=1$ which conjugates $\qth \arrowvert_{\field{T}}$ to $R_\theta$ is {\it quasisymmetric}. Recall that, given an increasing homeomorphism  $\eta: [0,\infty) \rightarrow [0,\infty)$, a mapping $h: A \subseteq \field{C}  \mapsto \field{C}$ is {\it $\eta$-quasisymmetric} if for each triple $z_0$, $z_1$, $z_2$ $\in A$
$${|h(z_0)-h(z_1)| \over |h(z_0)-h(z_2)|}  \le \eta \left( {|z_0-z_1| \over |z_0-z_2|} \right).$$  
A homeomorphism $h: \field{R} \mapsto \field{R}$ is $K$-quasisymmetric iff for every real $x$ and $\delta>0$
$$\left|{h(x+\delta)-h(x)   \over h(x) -h(x-\delta)     }\right| \le K.$$

 Let $H$ be some homeomorphic  extension of $h$ to the unit disk. One can, for example, choose the Douady-Earle extension of circle homeomorphisms (cf \cite{DE}), or Alhfors-Beurling extension (cf \cite{Leh}). In particular, the latter is quasiconformal; its constant of quasiconformality $M$ can be estimated in terms of the quasisymmetric constant $K$ of $h$ (cf \cite{Leh}) as
\begin{equation} \label{Kbound}
M \le 2 K-1.
\end{equation}

Define the {\it modified Blaschke product}
\begin{equation}\label{eq:modBP}
\bar{Q}_\theta(z)= \left\{  \qth(z),  \hspace{14.5mm} |z| \ge 1, \atop  H^{-1}(R_\theta(H(z))), \quad  |z| \le 1, \right.
\end{equation}
the two definitions matching along the boundary of $\field{D}$. $\bar{Q}_\theta$ is a degree $2$ branched covering of the sphere, holomorphic outside of $\field{D}$, and is quasiconformally conjugate on $\field{D}$ to a rigid rotation on the unit disk. We further conjugate $\bar{Q}_\theta(z)$ by a M\"obius transformation $m$
\begin{equation}\label{mbt}
m(z)={(1-\bar{a})(z-a) \over (1-a) (1-\bar{a}z)}, \quad a=H^{-1}(0),
\end{equation}
to place $H^{-1}(0)$ at the origin, and set $\tilde{Q}_\theta=m \circ \bar{Q}_\theta(z) \circ m^{-1}$.  Define the filled Julia set of $\tilde{Q}_\theta$ by
$$K(\tilde{Q}_\theta)=\left\{z \in \field{C}: {\rm the \quad orbit \quad of} \quad \{ \tilde{Q}_\theta^{\circ n} \}_{n \ge 0} \quad {\rm is \quad bounded}  \right\},$$
and the Julia set
$$J(\tilde{Q}_\theta)=\partial K(\tilde{Q}_\theta).$$

When the rotation number $\theta$ is irrational of bounded type, the action of $\tilde{Q}_\theta$ is conjugate to that of a quadratic polynomial. This follows from an argument due to Douady, Ghys, Herman and Shishikura (cf \cite{Do}) which we will now present.

\begin{figure}
\centering
\resizebox{0.65\textwidth}{!}{%
  \includegraphics{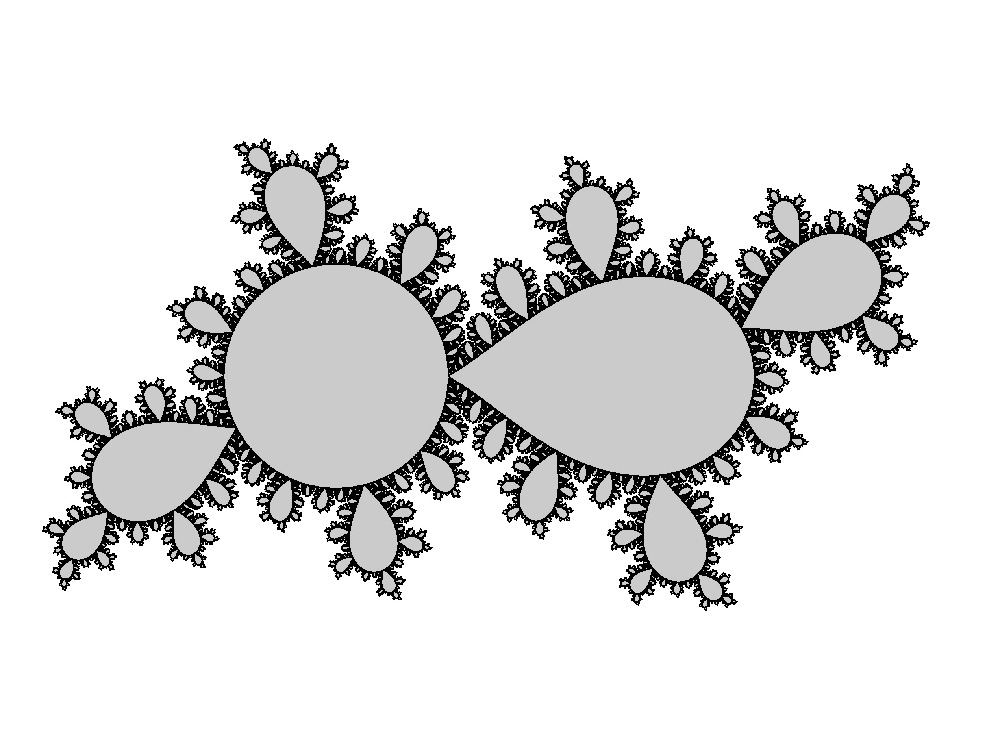}}
\caption{The filled Julia set of the modified Blaschke product with the golden mean rotation number.}
\label{fig:4}       
\end{figure}

Suppose that $0 < \theta < 1$ is an irrational of bounded type, and $H$ is a quasiconformal extension of the quasisymmetric conjugacy on the circle.  Recall that a {\it conformal structure} on a Riemann surface $S$  is conformal equivalence class of Riemann metrics on $S$. The {\it standard conformal structure} $\sigma_0$ on $\field{C}$ is a conformal equivalence class of the metric $d s^2=d x^2 +d y^2$. Define a new conformal structure on the plane, $\sigma_\theta$, invariant under $\tilde{Q}_\theta$, as the pull-back $H^* \sigma_0$ of $\sigma_0$ under $H$. Since $R_\theta$ preserves the standard conformal structure, $\tilde{Q}_\theta$ preserves $\sigma_\theta$ on $\field{D}$. Next, for every $n \ge 1$ pull $\sigma_\theta \arrowvert_\field{D}$ back by $\tilde{Q}_\theta^{\circ n}$ on the union off all $n$-th preimages of $\field{D}$ under $\tilde{Q}_\theta$, different from $\field{D}$. Notice, that since $\tilde{Q}_\theta$ is holomorphic outside of the unit disk, the dilatation of the pull-backs of $\sigma_\theta$ will not be increased. Finally, set $\sigma_\theta=\sigma_0$ outside of all preimages of $\field{D}$. Such $\sigma_\theta$ has a bounded dilatation and is $\tilde{Q}_\theta$ invariant. Therefore, by the Measurable Riemann Mapping Theorem (cf \cite{AB}, \cite{Bers} and \cite{Boy}), there exists a unique quasiconformal homeomorphism $f: \field{C} \mapsto \field{C}$, $f(\infty)=\infty$, $f(0)=0$ and $f(1)=1$, such that $f^* \sigma_0=\sigma_\theta$.  Set
$$f_\theta= f \circ \tilde{Q}_\theta \circ f^{-1}.$$

This $f_\theta$ is a self-map of the sphere that preserves $\sigma_0$, therefore it is holomorphic. It is, furthermore, a proper map of degree $2$ (since $\tilde{Q}_\theta$ is), therefore it is a quadratic polynomial. Since $f_\theta \arrowvert_{f(\field{D})}$ is quasiconformally conjugate to a rigid rotation, $f(\field{D})$ is contained in the  Siegel disk for $f_\theta$. Since $f(1)=1$ is a critical point of $f_\theta$,  $\overline{\{f_\theta^{\circ n}(1)  \}_{n \ge 0}}$    is the boundary of the Siegel disk, while $\{f_\theta^{\circ n}(1)  \}_{n \ge 0}$ itself is also dense in $f(\field{T})$, therefore $f(\field{T})$ is the boundary of the Siegel disk. Due to our normalization of  $f$, we must have
\begin{equation}
f_\theta=\qp.
\end{equation}

The above discussion is a sketch of the proof of the following theorem.
\begin{theorem} (Douady, Ghys, Herman, Shishikura)
Let  $f_\theta$ be a quadratic polynomial which has a fixed Siegel disk $\Delta_\theta$ of rotation number $\theta$ of bounded type. Then $f_\theta$ is quasiconformally conjugate to the modified  Blaschke product $\tilde{Q}_\theta$. In particular, $\partial \Delta_\theta$ is a quasicircle passing through the critical point of $f_\theta$.
\end{theorem}

\subsection{Bounds on the quasiconformal distortion of angles and eccentricities}
\label{crit_angle}

For $0<r<1$, let $\mu(r)$ be the modulus of the unit disk slit along the real axis from $0$ to $r$ -- the so-called {\it Gr\"otzsch's extremal domain}. The modulus function has the following explicit expression (cf \cite{Hertz})
\begin{equation}\label{mu}
\mu(r)={1 \over 4} {K'(r) \over K(r) },
\end{equation}
where $K(r)$ is the complete elliptic integral of the first kind
$$K(r)= \int_0^1 {d x \over \sqrt{1-x^2} \sqrt{1-r^2 x^2} },$$
and $K'(r)=K\left( \sqrt{1-r^2}  \right).$ In particular, the following asymptotic holds as $r \rightarrow 0$ (cf \cite{Hertz})
\begin{equation}\label{asympt_mu}
\mu(r)={1 \over 2 \pi} \ln {4 \over r} + O(r^2), \quad \mu(r) \le {1 \over 2 \pi} \ln {4 \over r},
\end{equation}
and
\begin{equation}\label{asympt_mu_inv}
\mu^{-1}(x)=4 e^{-2 \pi x}+O\left(e^{-4 \pi x} \right).
\end{equation}


Given $M \ge 1$, set
\begin{equation}\label{eq:phi_K}
\phi_M(r)=\mu^{-1}\left(M  \mu(r) \right).
\end{equation}
The ``distortion function'' $\phi_M$ is continuous and strictly increasing in $(0,1)$, with $\phi_M(0)=0$ and $\phi_M(1)=1$ (cf \cite{AgGe}). Furthermore, $\phi_M(t) \le t$.

The following result from \cite{AgGe}  will be important for our estimates
\begin{AGtheorem}  (S. B. Agard, F. W. Gehring) \label{AGthm}
 Suppose that $f$ is a $M$-quasiconformal mapping of the extended plane, and that $f(\infty)=\infty$. Then for each triple of distinct finite points $z_1$, $z_0$, $z_2$, 
$$\sin{ \beta \over 2} \ge \phi_M \left(\sin{\vartheta \over 2} \right),$$
where $\phi_M$ is as in $(\ref{eq:phi_K})$, and 
\begin{equation}
\label{alphabeta} \vartheta =\arcsin \left( { |z_1-z_2| \over |z_1-z_0|+|z_2-z_0|} \right), \quad  \beta  =\arcsin \left( { |f(z_1)-f(z_2)| \over |f(z_1)-f(z_0)|+|f(z_2)-f(z_0)|} \right).
\end{equation}
The inequality is sharp.
\end{AGtheorem}


We will end this subsection with a note about the relation between quasiconformality and quasisymmetry. We have already mentioned in  Subsection $\ref{blaschke}$, that a $K$-quasisymmetric map of a circle extends to the interior of the unit disk as a quasiconformal map whose constant of quasiconformality is at most $2 K-1$.  Conversely, $M$-quasiconformal maps, globally defined on $\field{C}$, are quasisymmetric with the quasisymmetric constant bounded in terms of $M$. The quantity that plays an important role in this bound is an estimate on distortion of eccentricities of circles under quasiconformal maps. A ``natural''  measure of the deviation of the image of a circle under an $M$-quasiconformal map from a round circle itself is the so-called ``circular distortion'' $\lambda(M)$, defined as 
\begin{equation}\label{def:circular_distortion}
\lambda(M)=\sup_{0 \le \phi \le 2 \pi} \{|f(e^{i \phi})|: f \in \cF \}, \quad \cF=\{f: \fC \mapsto \fC, \quad M-{\rm quasiconformal}, \quad f(0)=0, \quad f(1)=1\}
\end{equation}
(cf \cite{AIM}, page 81). It is known that 
\begin{equation}\label{eq:circular_dist_bound}
1 \le  \lambda(M) \le {1 \over 16} e^{\pi M}, \quad \lim_{M \rightarrow 1}\lambda(M)=1
\end{equation}
(cf \cite{AIM}, page 81). In particular, the circular distortion function is used to bound the constant of quasisymmetry in terms of the quasiconformal constant:
\begin{Atheorem} (K. Astala) 
If $f: \field{C} \mapsto \field{C}$ is  $M$-quasiconformal, and $z_0,z_1,z_2 \in \field{C}$,  then
$${|f(z_0)-f(z_1)| \over |f(z_0)-f(z_2)|}  \le \eta \left( {|z_0-z_1| \over |z_0-z_2|} \right),$$  
where
$$\eta(t)= \lambda(M)^{2 M} \max\{t^M, t^{1 / M}\}, \quad t \in [0, \infty),$$ 
and $\lambda(M)$ is the circular distortion defined in $(\ref{def:circular_distortion})$.
\end{Atheorem}

\subsection{The quasisymmetric property  and moduli}

Given any four points $a$, $b$, $c$ and $d$ on the real line, ordered as $a \le b \le c \le d$  or $a \ge b \ge c \ge d$, define their cross-ratio as 
$${\bf Cr}(a,b,c,d)={|a-b| \cdot |c-d| \over |a-c| \cdot |d-b|  }.$$

An important property of the cross-ratio is that it is preserved under linear-fractional transformations.

Following \cite{Sw}, we introduce a more general class of functions which play a role similar to the cross-ratio.
\begin{definition}\label{crmodulus}
A cross-ratio modulus is a function $\chi$ from all quadruples of points on the real line which satisfy $a \le b \le c \le d$, or $a \ge b \ge c \ge d$, with values in $[0, \infty)$, provided that
\begin{itemize}
\item[-] there is a constant $C$ such that if ${\bf Cr}(a,b,c,d) \ge 1/4$, then $\chi(a,b,c,d) \ge C$;
\item[-] for every $\epsilon >0$ there is a $\delta>0$ so that if ${\bf Cr}(a,b,c,d) < \delta$, then $\chi(a,b,c,d) < \epsilon$.
\end{itemize}
\end{definition}

\begin{example}
Consider a Jordan curve $\gamma \in \hat{\field{C}}$ with four ordered points on it, $A,B,C$ and $D$. The four points divide the curve into four arcs. Such configuration defines a {\it quadrilateral} $Q(A,B,C,D)$. Choose one of the arcs to be the {\it base}, and map one of the components of the complement of $\gamma$ conformally onto the rectangle with the vertices $(0,1,1+i a, i a)$, so that the base is mapped onto $(0,1)$. The parameter $a$ is called the modulus of the quadrilateral, and denoted $\mod Q(A,B,C,D)$. 

In the particular case when $\gamma=\field{R}$, set
\begin{equation}\label{qmodulus}
\chi(A,B,C,D)={1 \over \mod{Q(A,B,C,D)}}.
\end{equation}
If $A=B$ or $C=D$, set $\chi(A,B,C,D)=0$.

The modulus of a quadrilateral is related to the modulus of the {\it ring domain}. Let $A$, $B$, $C$ and $D$ be points on the real line, and let $v$ be a Jordan arc inside the closed upper half plane connecting $A$ and $D$. Intervals $(A,B)$, $(B,C)$, $(C,D)$ and the arc $v$ define a quadrilateral $Q'$ with the base $(A,B)$. Let $Q''$ be its reflection with respect to the real line. The union of the interiors of $Q'$, $Q''$ and the intervals $(A,B)$ and $(C,D)$ define the ring domain of $R(Q')$ {\it associated with $Q'$}. The moduli of a quadrilateral and of the associated ring domain are related as 
$$\mod{R(Q')}={1 \over 2 \mod{Q'}}.$$
This can be seen from the following argument. Map the ring domain $R(Q')$ conformally to an annulus with the radii $1$ and $\rho>1$, whose modulus is $(2 \pi)^{-1} \ln \rho$, so that the interval $(A,B)$ gets mapped to the interval $(-\rho, -1)$. The same conformal transformation maps the quadrilateral $Q'$ to the upper semi-annulus. A further application of the logarithm, maps the semi-annulus to a quadrilateral $(0,\ln \rho, \ln \rho +i \pi, i \pi)$ (notice, that the base $(A,B)$ is mapped to $(i \pi, \ln \rho +i \pi)$). The modulus of the last rectangle is $\pi / \ln \rho$.

We will now describe how the moduli of the quadrilaterals $Q(A,B,C,D)$ and $Q'$ are related. Consider the set of {\it admissible curves} - all curves  contained inside the quadrilateral and joining a point in the base with a point in the opposite side. Let $Q_1$ and $Q_2$ be any two quadrilaterals. If every admissible curve for a quadrilateral $Q_1$ is contained in some admissible curve of $Q_2$, then $\mod Q_1 \ge \mod Q_2$.  Hence,  $\mod Q' \ge \mod Q(A,B,C,D)$.  On the other hand, one can also bound from above the modulus of $Q'$ in terms of that of $Q(A,B,C,D)$:
\begin{lemma} (Lemma 2.1 in \cite{Sw}) \label{2.1}
Suppose that points $A,B,C,D$ are cyclically ordered on $\field{R}$. For every $k >0$, if $|D-A| \le k/2$, then there is a Jordan arc $\gamma$ in the closed upper-half plane with the endpoints $D$ and $A$, such that the quadrilateral $Q'$ with sides $(A,B)$, $(B,C)$, $(C,D)$ and $\gamma$ has the interior disjoint from $\{z \in \field{C}: \Im {z} \ge k  \}$, and satisfies
$$\mod Q' \le {1 \over 1-{|D-A| \over k}   } \mod Q(A,B,C,D).$$
\end{lemma}

We will now show  that $\chi$ in $(\ref{qmodulus})$ is indeed a cross-ratio modulus. 

Let $Q'$ be the quadrilateral with the sides $(A,B)$, $(B,C)$, $(C,D)$ and a semicircular arc $\gamma$ connecting points $A$ and $D$. If  ${\bf Cr}(A,B,C,D) \ge 1/4$ then 
$$\left(1+{|B-C| \over |A-B|} \right) \left(1+{|B-C| \over |C-D|} \right) \le 4 \implies |A-B| \ge {1 \over 3} |B-C| \quad {\rm and}  \quad  |C-D| \ge {1 \over 3} |B-C|,$$
and, therefore, the ring domain $R(Q')$ contains an annulus with modulus at least 
$${1 \over 2 \pi} \ln{{1 \over 2}+{1\over3} \over {1 \over 2}} ={1 \over 2 \pi} \ln {5 \over 3}.$$
Since $\mod Q(A,B,C,D) \le \mod Q'$, we get 
\begin{equation}\label{chi_const}
\chi(A,B,C,D)={1 \over \mod Q(A,B,C,D)} \ge {1 \over \mod Q'}= 2 \mod R(Q')  \ge {1 \over \pi} \ln {5 \over 3}.
\end{equation}

The second condition in the definition of the cross-ratio modulus follows from the above Lemma and  the following Teichm\"uller's Modulus Theorem (cf page 56 in \cite{LeVi}). 
\begin{Ttheorem}\label{tt}
If a ring domain $R$ separates points $0$ and $z_1$ from  $z_2$ and  $\infty$, then 
$$\mod R \le 2 \mu \left( \sqrt{|z_1| \over |z_1|+|z_2|}  \right),$$
where $\mu$ is as in $(\ref{mu})$. 
\end{Ttheorem}

Now, we can demonstrate the  second property of the cross-ratio modulus. Assume that  ${\bf Cr}(A,B,C,D)< \delta$. This inequality implies that 
\begin{equation}\label{eps}
\min\left\{ {|A-B| \over |C-B|}, {|C-D| \over |C-B|} \right\} <  \delta^{1 \over 2}  /(1-\delta^{1 \over 2}) \equiv\delta'.
\end{equation}
By Lemma $\ref{2.1}$, there exists a quadrilateral $Q'$ with the base $(A,B)$, sides $(B,C)$, $(C,D)$, and an  arc $\gamma$ connecting $A$ and $D$, disjoint from $\{z \in \field{C}: \Im {z} \ge 2 |D-A|  \}$, whose modulus satisfies
$$2 \mod Q \ge  \mod Q'.$$
By  Teichm\"uller's Modulus Theorem,
$$\mod R(Q') < 2 \mu \left( \sqrt{1 \over 1+\delta'}  \right),$$
and
\begin{equation}
\label{cond2} \chi(A,B,C,D)={1 \over \mod Q} \le { 2  \over \mod Q'} \le 4 \mod R(Q') \le 8   \mu \left( \sqrt{1 \over 1+\delta'}  \right)=8   \mu \left( \sqrt{1 -\delta^{1 \over 2}}  \right) \equiv \epsilon.
\end{equation}
\end{example}

A configuration of $n$ quadruples of points  $(a_i,b_i,c_i,d_i)$, $i=1, \ldots, n$, will be called {\it allowable}, if the intervals $(a_i,d_i)$ are pairwise disjoint modulo $1$ and $d_i-a_i <1$. We will also say that a configuration of $n$ quadruples of points  $(a_i,b_i,c_i,d_i)$, $i=1, \ldots, n$, has the {\it intersection number} $k$  if the supremum over points $x \in \field{R}$ of $k'$ such that $x$ is contained modulo $1$ in $k'$ intervals from the configuration, is equal to $k$.

\begin{definition}\label{cross-ratio-def}
Let $g: \field{R} \mapsto \field{R}$ be strictly increasing and suppose that $g(x)-x$ is $1$-periodic. Let $\chi$ be a cross-ratio modulus. We say that $g$ satisfies the cross-ratio inequality with respect to $\chi$ with bound $Q$ iff for any choice of quadruples of points $(a_i,b_i,c_i,d_i)$, $i=1, \ldots, n$, in an allowable configuration the following holds:
$$\prod_{i=1}^n {\chi(g(a_i),g(b_i),g(c_i),g(d_i)) \over \chi(a_i,b_i,c_i,d_i) } \le Q.$$ 
\end{definition}

We mention the following Lemma without the proof.

\begin{lemma} (Lemma 1.1. from \cite{Sw}) \label{1.1}
Suppose that $g$ satisfies the cross-ratio inequality with respect to $\chi$ with the bound $Q$. Let $(a_i, b_i, c_i, d_i)$ be any configuration of quadruples of points with the intersection number $k$. Then
$$\prod_{i=1}^n {\chi(g(a_i),g(b_i),g(c_i),g(d_i)) \over \chi(a_i,b_i,c_i,d_i) } \le Q^{2 k}.$$ 
\end{lemma}

\section{Statement of results} \label{statements}

We are now ready to state the main result of the paper in a more detailed form. 

\begin{mainthm} \label{maintheorem}
Suppose that $\theta=[a_1, a_2, \ldots ]$, $a_i < B$, is a quadratic irrational whose continued fraction is preperiodic with period $s$, and suppose that $\{p_n / q_n\}$ is the sequence of the best rational approximants of $\theta$. Then there exists $C(n)>1$, with $\lim_{n \mapsto \infty} C(n)=1$, such that the following holds.

\medskip

\noindent 1) If $s$ is odd, then
\begin{equation}\label{b1.1}
\alpha^{\gamma}  \le {|\qp^{q_{n+s}}(1)-1| \over |\qp^{q_{n}}(1)-1|} \le {C(n) \over \sqrt{K^{-1} \left(1 \over 1+K  \right)^{\log_2{\alpha^2 \over 1-\alpha^2}+1}+1}}, \quad {\rm if} \quad  \alpha  > {1 \over \sqrt{2}},
\end{equation}
and
\begin{equation}\label{b1.2}
\alpha^{\gamma} \le {|\qp^{q_{n+s}}(1)-1| \over |\qp^{q_{n}}(1)-1|}  \le C(n) \left(1+K^{-1} \right)^{-{1 \over 2} [s \log_2 \vartheta]}, \quad {\rm if} \quad  \alpha  \le {1 \over \sqrt{2}}.
\end{equation}
\medskip
\noindent 2) If $s$ is even, then
\begin{equation}\label{b2.1}
\alpha^{\gamma}  \le {|\qp^{q_{n+s}}(1)-1| \over |\qp^{q_{n}}(1)-1|} \le {C(n) \over K^{-1} \left(1 \over 1+K  \right)^{\log_2{\alpha \over 1-\alpha}+1}+1}, \quad {\rm if} \quad  \alpha  > {1 \over 2}, 
\end{equation}
and
\begin{equation}\label{b2.2}
\alpha^{\gamma} \le {|\qp^{q_{n+s}}(1)-1| \over |\qp^{q_{n}}(1)-1|}  \le C(n) \left(1+K^{-1} \right)^{-[s \log_2 \vartheta]}, \quad {\rm if} \quad  \alpha  \le {1 \over 2}.
\end{equation}
In the above bounds, $\vartheta=\alpha^{-{1 \over s}}$,
\begin{equation}\label{eq:Ks} 
K=\lambda(2K_2-1)^{4 K_2-2} K_2^{2 K_2-1}, \quad K_2=\max\left\{2, K_1^{B+1} \right\}, \quad K_1=2 \left(1-\left(\mu^{-1}\left( {1 \over 8^7 \pi} \ln {5 \over 3}   \right)   \right)^2 \right)^{-6},
\end{equation}
where $\mu$ is the modulus $(\ref{mu})$  of the Gr\"otzsch's extremal domain, $\lambda(M) \le e^{\pi M}/16$ is the circular distortion defined in $(\ref{def:circular_distortion})$, and
$$\gamma=1-{16 \over \pi} \left( { \sqrt{2}(\sqrt{3}-1) \over 16} \right)^{2 K_2-1 \over 2} \left(1+O \left( \left( { \sqrt{2}(\sqrt{3}-1) \over 16} \right)^{2 K_2-1} \right) \right).$$
\end{mainthm}

\section{An upper bound on the scaling ratio}

The following is a version of the H\"older continuity property for quasisymmetric homeomorphisms. The H\"older property is a classical result (cf \cite{Ahl}, \cite{deFaria}); here we will prove a Lemma adopted for our situation.

\begin{lemma} \label{SRbounds}

Let $T$ be an interval in $\field{R}$, and $h: T \mapsto h(T)$ be a quasisymmetric homeomorphism with a quasisymmetric constant $K$. 

\noindent 1) If $I \subset J$ are two intervals in $T$ sharing a single boundary point $c$, and satisfying $|I| / |J|=\alpha \le 1/2$, then
$$\left({1 \over 1+K}\right)^{[s \log_2 \vartheta ]+1}  \le {|h(I) | \over  |h(J)|}  \le \left({1 \over 1+K^{-1}} \right)^{[s \log_2 \vartheta]}, \quad \vartheta=\alpha^{-{1 \over s}}.$$

\noindent 2) If $I$ and $J$ are two closed intervals in $T$ such that their intersection is a single boundary point $I \cup J=\{c\}$, and satisfying $|I| / |J|=\alpha<1$, then
$$ {1 \over (1+K)^{\left[s \log_2 {\nu}\right]+1}-1  } \le {|h(I)| \over |h(J)|} \le {1 \over (1+K^{-1})^{\left[s \log_2 \nu \right]}-1  }, \quad \nu=\left( {1+\alpha \over \alpha} \right)^{{1 \over s}}.$$

\noindent 3) Let $I \subseteq J$ be two intervals in $T$ sharing a single boundary point $c$, such that $|I| / |J|=\alpha > 1/2$. Let $\tilde{I} \supset J \setminus I$ be a closed interval of length $|I|$ sharing one endpoint with $I$  and  $J \setminus I$. Suppose that $\tilde{I} \subset T$,  then
$$ {1 \over K \left(1 \over 1+K^{-1}  \right)^{\left[ \log_2{\alpha \over 1-\alpha}\right] }+1} \le {|h(I)| \over |h(J)|} \le {1 \over K^{-1} \left(1 \over 1+K  \right)^{\log_2{\alpha \over 1-\alpha}+1} +1}.$$

\end{lemma}
\begin{proof}
{\it 1) Case $I \subset J$}. Let $J_n$ be the unique closed subinterval of $J$ of length $2^{-n} |J|$ that shares the same end point with $I$ and $J$.  Since $h$ is quasisymmetric on $J$ with constant $K$, we have
\begin{equation}
{|h(J_{n+1})| \over |h(J_n)|} \le  {|h(J_{n+1})| \over |h(J_{n+1})|+ |h(J_n \setminus J_{n+1})|  }   \le {1 \over 1+ { |h(J_n \setminus J_{n+1})| \over |h(J_{n+1})| } }  \le {1 \over 1+K^{-1}}, 
\end{equation}
in a similar way,
$${1 \over 1+K} \le {|h(J_{n+1})| \over |h(J_n)|}.$$
 
We have $|h(J_n) | / |h(J)|=\prod_{i=0}^{n-1} |h(J_{i+1})| / |h(J_i)|$, therefore,
$$\left({1 \over 1+K}\right)^n \le {|h(J_n) | \over  |h(J)|} \le \left({1 \over 1+K^{-1}} \right)^n.$$

Now, set $m=[\log_2 \alpha^{-1}]+1$, then $J_m \subseteq I \subseteq J_{m-1}$, therefore,
$$\left({1 \over 1+K}\right)^{[-\log_2 \alpha ]+1} \le {|h(J_m) | \over  |h(J)|} \le {|h(I) | \over  |h(J)|} \le {|h(J_{m-1}) | \over  |h(J)|} \le \left({1 \over 1+K^{-1}} \right)^{[-\log_2 \alpha]}.$$

{\it 2) Case $I \cap J=\{c\}$}. 
$$
{|h(I)| \over |h(J)|}= {|h(I)| \over |h(J \cup I)|- |h(I)|}  =  {1 \over {|h(J \cup I)| \over |h(I)|}-1  }, 
$$
and, according to part $1)$,
$$ {1 \over (1+K)^{\left[-\log_2 {\alpha \over 1+\alpha}\right]+1}-1  } \le {|h(I)| \over |h(J)|} \le {1 \over (1+K^{-1})^{\left[-\log_2 {\alpha \over 1+\alpha}\right]}-1  } $$

{\it 3) Case $I \subseteq J$, $|I|/|J|=\alpha > 1/2$.}  Let $\tilde{I}$ be a closed subinterval of $I$ of length $|J \setminus I|$ such that the intersection of $\tilde{I}$ with $J \setminus I$ is a single boundary point of both intervals.
\begin{eqnarray}
\nonumber {|h(I)| \over |h(J)|} &=&{1 \over {|h(J \setminus I)| \over |h(I)|}+1} = {1 \over {|h(J \setminus I)| \over |h(\tilde{I})|}{|h(\tilde{I})| \over  |h(I)|} +1} \ge {1 \over K \left(1 \over 1+K^{-1}  \right)^{\left[ \log_2{\alpha \over 1-\alpha} \right]}+1}, \\
\nonumber {|h(I)| \over |h(J)|} &=& {1 \over {|h(\hat{I})| \over |h(I)|}{ |h(J \setminus I)| \over |h(\hat{I})| }+1} \le {1 \over K^{-1} \left(1 \over 1+K  \right)^{\log_2{\alpha \over 1-\alpha}+1} +1}.
\end{eqnarray}
\end{proof}
\qed

With regard to critical circle maps with a rotation number whose continued fraction is preperiodic, a particular case of which is the dynamical system $\qp \arrowvert_{\partial \sd}$, the above  Lemma essentially provides bounds on the scaling ratio, of the form $C_1 \delta_1^s \le |\lambda| \le C_2 \delta_2^s$. This can be seen as follows. Recall, that according to our discussion of the Blaschke product model of the Siegel disk in Section $\ref{blaschke}$, 
$$\qp=f \circ m \circ h^{-1} \circ \rot \circ h \circ  m^{-1} \circ f^{-1}$$ 
on $\partial \sd$, where $\rot$ is the rigid rotation by angle $\theta$, $f$ is the quasiconformal conjugacy of the modified Blaschke product $\tilde{Q}_\theta$ with the quadratic polynomial $\qp$, $h$ is the quasisymmetric (by the result of Herman and 'Swi{\c{a}}tek) conjugacy of $\rot$ with the Blaschke product $Q_\theta$ on the unit circle, and $m$ is an appropriately chosen M''obeus transformation. Now, take $J$ and $I$ to be the intervals $[1, \rot^{q_n}(1)]$ and $[1,\rot^{q_{n+s}}(1)]$. Then, Lemma $\ref{SRbounds}$ gives bounds on the ratios $[1, \qp^{q_{n+s}}(1)]/[1,\qp^{q_{n}}(1)]$ in terms of the quasisymmetric constant of $f \circ m \circ h^{-1}$.

The only remaining ingredient is a bound on the constant of quasisymmetry. This will be the subject of Section $\ref{qc_constant}$, where we will also show that this bound can be taken essentially independent of the rotation number. 

We notice, however, that for sufficiently large $K$, the lower bound on the quasisymmetric distortion of a ration of two intervals is of the order $\left(1+K \right)^{[-\log_2 \alpha ]+1}$ in Cases 1) and 2) in the above Lemma. But
$$\left({1 \over 1+K}   \right)^{[\log_2 \alpha^{-1}]+1} \le (1+K)^{\log_2 \alpha}=\alpha^{1 \over \log_{1+K}2},$$ 
the power of $\alpha$ in the last expression being larger than $1$. This means, that the lower bounds in the Lemma above are worse than the bound $|\lambda|\ge \alpha$ of Buff and Henriksen. In the next Section we will derive a better lower bound on the absolute value of the scaling ratio of the form $\alpha^{\gamma}$ with a bound on  $0<\gamma <1$.

\section{A lower bound on the scaling ratio}
\label{lower_bound}

Recall, that by the result of McMullen \cite{McM2} the rescalings of the Siegel disk and its preimage converge to $\Lambda$-invariant quasidisks $\cD$ and $\cD'$.
 
Consider the cylinder $C=\cD / \Lambda^2$. Let $f$ be the quasiconformal conjugacy between the dynamics of the modified Blaschke product $\tilde{Q}_\theta$ and that of $\qp$ on the sphere, as described in Section \ref{blaschke}.

 Let $\gamma_{{\min}}<\gamma_{{\max}}$ be any two angles such that the angle $\gamma$  between any two vectors $f(z_1)-1$ and $f(z_2)-1$, $z_i \ne 1$, of equal length, $|f(z_1)-1|=|f(z_2)-1|$, and with the end points $f(z_i) \in \partial \sd$ on ``opposite sides'' of the critical point, i.e. ${\rm signum}\{\Im{(z_1)}\} \ne {\rm signum}\{\Im{(z_2)}\}$, is asymptotically (that is, as $z_i \rightarrow 1$) bounded by $\gamma_{{\min}}$ and  $\gamma_{{\max}}$. The boundary of the Siegel disk is, therefore, asymptotically contained between the two curvilinear sectors  $\cS_{\min}$ and $\cS_{\max}$  with the vertex at $1$ and angles $\gamma_{\min}$ and $\gamma_{\max}$, respectively. Consider the images $\cT_{\min}$ and $\cT_{\max}$  of these sectors under the map $\ln{(z-1)}$: the two curvilinear strips  $\cT_{\min}$ and $\cT_{\max}$  whose height in every vertical section is $\gamma_{\min}$ and $\gamma_{\max}$, respectively  (see Fig. $\ref{fig:5}$).  Set
$$C_{{\max}}=\cT_{\gamma_{\max}} / L, \quad C_{{\min}}=\cT_{\gamma_{\min}} / L,$$
where $L(z)=z+\ln{\lambda \bar{\lambda}}$ if the period $s$ is odd, and $L(z)=z+\ln{\lambda^2}$ if $s$ is even.

By Rengel's inequality (see  e.g. Section $4.3$ of \cite{LeVi}) the width $-\ln{|\lambda|^2}$ of the fundamental domains for $C_{{\max}}$ and $C_{{\min}}$ satisfies
$$\left(-\ln{|\lambda|^2} \right)^2 \le { {\rm area}\  C_{\min} \over \mod C_{\min}} \ {\rm and} \ \left( -\ln{|\lambda|^2} \right)^2 \le { {\rm area} \ C_{\max} \over \mod C_{\max}}.$$ 
At the same time, since the height of  $\cT_{\min}$ and $\cT_{\max}$ in every vertical section is exactly $\gamma_{\min}$ and $\gamma_{\max}$, the areas of the fundamental domains are exactly equal to the areas of the parralelograms with one (vertical) side equal to $\gamma_{\min}$ or $\gamma_{\max}$, and the distance between the vertical sides  $-\ln{|\lambda|^2}$, and, therefore,
\begin{eqnarray}
\label{minbound} \left( \ln {1 \over |\lambda|^2} \right)^2 \le  {\gamma_{{\min}}  \ln {1 \over |\lambda|^2}  \over \mod C_{\min}}, \\
\label{maxboundd}\left( \ln {1 \over |\lambda|^2} \right)^2  \le {  \gamma_{{\max}}  \ln {1 \over |\lambda|^2}  \over \mod C_{\max}}.
\end{eqnarray}
We can now weaken the second of these inequalities (but not the first!), by using $\mod C \le \mod C_{\max}$, and, then, use the exact value of $\mod C$ from Lemma $\ref{BHlemma}$:
\begin{eqnarray}
\label{maxbound} \ln {1 \over |\lambda|^2}  \le {  \gamma_{{\max}}   \over \mod C_{\max}}  \le { \gamma_{\max}  \over \mod C}={ \gamma_{\max}  \over -{\pi  \over \ln \alpha^2}} \implies |\lambda| \ge \alpha^{\gamma_{{\max}} \over \pi}.
\end{eqnarray}

The direction of the inequality $(\ref{minbound})$ does not allow us to obtain an {\it upper bound} on $|\lambda|$ of the form $|\lambda| \le \alpha^{\gamma_{{\min}} \over \pi}$; we will estimate $|\lambda|$ from above  using the upper bounds  from Lemma $\ref{SRbounds}$.

\begin{figure}
\centering
{\includegraphics[width=5cm,angle=-90]{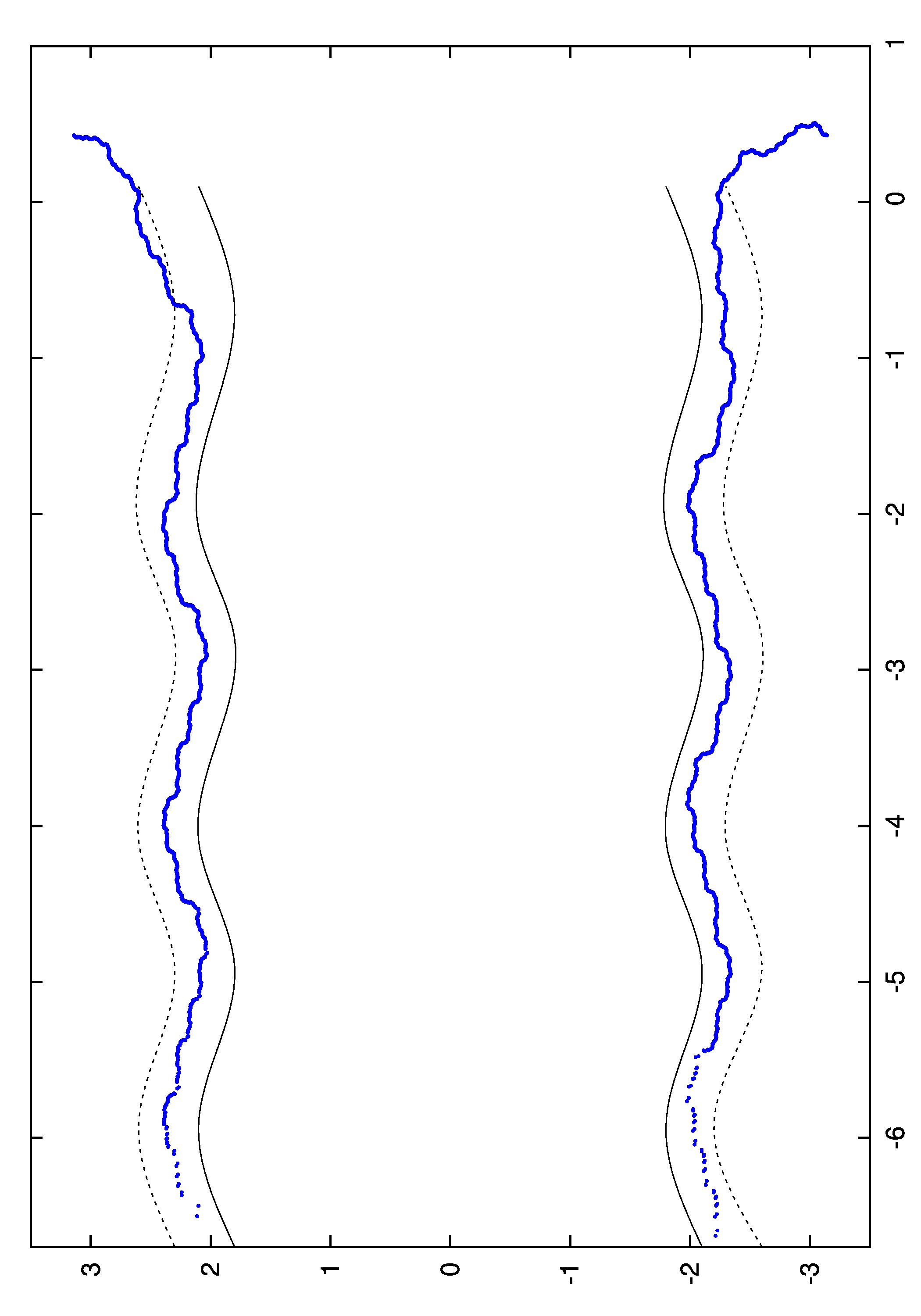}}
\vspace{3mm}       
\caption{The orbit of the critical point of the quadratic polynomial with the rotation number $[5,5,5, \ldots]$ in the neighborhood of the critical point in the logarithmic coordinate $\ln(z-1)$. The strip $\cT_{\min}$ is bounded by two solid lines, its height being $\gamma_{\min}$, the strip between the dashed lines is  $\cT_{\min}$, its height - $\gamma_{\max}$. The orbit is asymptotically contained between the lines.}
\label{fig:5}       
\end{figure}

We will now proceed with an estimate  on $\gamma_{{\max}}$.

Recall, that the modified Blaschke product $\tilde{Q}_\theta$ (see $(\ref{eq:modBP})$ and the following definition of $\tilde{Q}_\theta$) is a degree two map that preserves the circle $\field{T}$. The two preimages of $\field{T}$ under $\tilde{Q}_\theta$ are $\field{T}$ itself and a topological circle that lies in $\hat{\field{C}} \setminus \bar{\field{D}}$ and and touches $\field{T}$ at the critical point $1$ (``one-half of figure eight'', see Fig. $\ref{fig:4}$).  We will take  $z_1 \in \field{T}$ and $z_2$ in this other preimage of $\field{T}$ under the modified Blaschke product  with $\Im(z_1)>0$ and $\Im(z_2)>0$, both sufficiently close to $z_0=1$, and obtain a lower bound on the angle $\beta'$ between the vectors $f(z_1)-1$ and $f(z_2)-1$. An upper bound on the angle $\gamma$, then, is $\pi$ minus the lower bound on the angle $\beta'$. An upper bound on $\gamma$, in turn, produces a lower bound on $\lambda$ through formula $(\ref{maxbound})$.

Given a triple of points $z_0=1,z_1, z_2 \in \field{C}$,  let $\vartheta$ and $\beta$ be as in  the Quasiconformal Disstortion of Angles Theorem on page 9.

Elementary geometric considerations demonstrate
\begin{eqnarray}
\nonumber {|z_1-z_2| \over |z_1-1|+|z_2-1|}&=&{ \sqrt{ |z_1-1|^2+|z_2-1|^2 -2 |z_1-1||z_2-1| \cos{\pi \over 3} } \over  |z_1-1|+|z_2-1| }+O\left(\left(|z_1-1|+|z_2-1|\right)^2 \right) \\
\nonumber &=&{ \sqrt{ |z_1-1|^2+|z_2-1|^2 - |z_1-1||z_2-1| } \over  |z_1-1|+|z_2-1| }+O\left(\left(|z_1-1|+|z_2-1|\right)^2 \right) \\
\label{bound1} &=&{ \sqrt{ z^{-2}+z^2 -1 } \over  z^{-1}+z }+O\left(\left(|z_1-1|+|z_2-1|\right)^2 \right),
\end{eqnarray}
where 
$$z=\sqrt{|z_2-1| \over |z_1-1|}.$$
The minimum of the leading term in $(\ref{bound1})$ is achieved at $z=1$, and is equal to $1/2$; therefore, the lower bound on the inverse sine of $(\ref{bound1})$ is equal to 
$$\vartheta={\pi \over 6}+O\left(\left(|z_1-1|+|z_2-1|\right)^2 \right).$$
Next,
$$\sin{\vartheta \over 2}=\sin{{\pi \over 12}+O\left(\left(|z_1-1|+|z_2-1|\right)^2 \right)}={1 \over 4}\sqrt{2}(\sqrt{3}-1)+O\left(\left(|z_1-1|+|z_2-1|\right)^2 \right).$$

At the next step we will evaluate the function $\phi_M$ at the above value. For brevity of notation, denote
$$\zeta={\sqrt{2}(\sqrt{3}-1) \over 4}.$$
According to $(\ref{asympt_mu})$, 
$$ \mu \left(\zeta + O\left(\left(|z_1-1|+|z_2-1|\right)^2 \right) \right)   \le  -{1 \over 2 \pi} \ln \left( { \zeta \over 4} \right) + O\left(\left(|z_1-1|+|z_2-1|\right)^2 \right).
$$
We can now use the definition $(\ref{eq:phi_K})$ of $\phi_M$, together with $(\ref{asympt_mu_inv})$, to obtain
$$\phi_M\left(\zeta +O\left(\left(|z_1-1|+|z_2-1|\right)^2 \right) \right) \ge \left( 4 \left( \zeta \over 4 \right)^M +O \left( \left(  {\zeta \over 4} \right)^{2 M}   \right) \right) \left(1+O\left(M \left(|z_1-1|+|z_2-1|\right)^2 \right) \right),
$$
that is, according to the Quasiconformal Disstortion of Angles Theorem on page 9,
\begin{eqnarray}
\nonumber \beta &\ge &  2 \arcsin \left( \left( 4 \left( \zeta \over 4 \right)^M +O \left( \left(  {\zeta \over 4} \right)^{2 M}   \right) \right) \left(1+O\left(M \left(|z_1-1|+|z_2-1|\right)^2 \right) \right) \right) \implies \\
\nonumber  |f(z_1)-f(z_2)| & \ge&   \left( |f(z_1)-f(1)|+|f(z_2)-f(1)| \right)  \sin \left(  2 \arcsin \left( \left( 4 \left( \zeta \over 4 \right)^M +O \left( \left(  {\zeta \over 4} \right)^{2 M}   \right) \right) \right. \right.  \\
\nonumber &\phantom{a}& \phantom{aaaaaaaaaaaaaaaaaaaaaaaaaaaaaaaa} \times \left. \left. \left(1+O\left(M \left(|z_1-1|+|z_2-1|\right)^2 \right) \right) \right)  \right) \\
\nonumber  & \ge&   \left( |f(z_1)-1|+|f(z_2)-1| \right)  \left( 8 \left( {\zeta \over 4 } \right)^M + O\left(  \left( {\zeta \over 4 } \right)^{2 M}  \right)  \right)  \left(1+O\left(M \left(|z_1-1|+|z_2-1|\right)^2 \right) \right).
\end{eqnarray}
Recall that $\arccos$ is a decreasing function of its argument, therefore, when $|f(z_1)-1|=|f(z_2)-1|$ (the condition which enters the definitions of angles $\gamma_{\max}$ and $\gamma_{\min}$ in the beginning of Section $\ref{lower_bound}$), we obtain the following lower bound on the angle $\beta'$ between the vectors $f(z_1)-1$ and $f(z_2)-1$: 
\begin{eqnarray}
\nonumber &&\arccos \left({|f(z_1)-1|^2+|f(z_2)-1|^2-|f(z_1)-f(z_2)|^2  \over 2  |f(z_1)-1| |f(z_2)-1|    } \right)   \ge  \\
\nonumber &&\arccos \left({|f(z_1)-1|^2\!+\!|f(z_2)-1|^2 \over  2  |f(z_1)-1| |f(z_2)-1|  } -\right. \\
\nonumber &&\phantom{aaaaaaaaaa} -\left.{\left( |f(z_1)\!-\!1|\!+\!|f(z_2)\!-\!1| \right)^2  \left( 8 \left( {\zeta \over 4 } \right)^M \!+\! O\left(  \left( {\zeta \over 4 } \right)^{2 M}  \right)  \right)^2  \left(1\!+\!O\left(M \left(|z_1-1|+|z_2-1|\right)^2 \right) \right)^2 \over 2  |f(z_1)-1| |f(z_2)-1|    } \right)   \ge  \\
\nonumber &&\phantom{aaa} \arccos \left( {1 \over 2} \left( 2-  4    \left( 8 \left( {\zeta \over 4 } \right)^M + O\left(  \left( {\zeta \over 4 } \right)^{2 M}  \right)  \right)^2  \left(1+O\left(M \left(|z_1-1|+|z_2-1|\right)^2 \right) \right)^2 \right) \right) \ge  \\
\nonumber &&\phantom{aaa} \arccos \left( 1- \left( 128 \left( {\zeta \over 4 } \right)^{2 M} + O\left(  \left( {\zeta \over 4 } \right)^{3  M}  \right)  \right)   \left(1+O\left(M \left(|z_1-1|+|z_2-1|\right)^2 \right) \right)^2 \right).
\end{eqnarray}

We finally take the limit $z_i \rightarrow 1$, and get that
\begin{eqnarray}
\nonumber \gamma_{\max} &=& \pi-\beta'=\pi-\arccos \left( 1- \left( 128 \left( {\zeta \over 4 } \right)^{2 M} + O\left(  \left( {\zeta \over 4 } \right)^{3  M}  \right)  \right) \right)\\
 \label{gamma_max} &=&\pi-\sqrt{2} \sqrt{ 128 \left( {\zeta \over 4 } \right)^{2 M} + O\left(  \left( {\zeta \over 4 } \right)^{3  M}  \right)   }=\pi- 16 \left( {\zeta \over 4 } \right)^{M}\left(1 + O\left(  \left( {\zeta \over 4 } \right)^{ M}  \right) \right). 
\end{eqnarray}

\section{A bound on the quasisymmetric constant} \label{qc_constant}

In this section we will give an ouline of \'Swi\c{a}tek's proof of the quasisymmetric property of the conjugacy of the critical circle maps to the rotation from \cite{Sw}. The goal of our outline will be to recover a useful expression for the constant of quasisymmetry of the conjugacy only in terms of the upper bound on the integers in the continued fraction expansion of the rotation number. We will not give a detailed proof of the \'Swi\c{a}tek's theorem, our objective will be to extract the necessary dependency.

To obtain a bound on the quasisymmetric constant of the homeomorphism that conjugates a circle homeomorphism with a critical point to the rigid rotation, we will have to obtain several commensurability relations between the intervals  bounded by a point and its forward and backward closest returns. This will be done in Lemma $\ref{1.2}$. The proof of this Lemma uses the following fact (see, for example, \cite{Sw})
\begin{lemma}\label{f1.1}
Let $g: \field{R} \mapsto \field{R}$ be increasing and $1$-periodic. Choose $Q$ positive so that $\rho(g)=P/Q$ in simplest terms, if $\rho(g)$ is rational, or $Q=\infty$, if $\rho(g)$ is irrational. Then for every $x \in \field{R}$, and every pair of integers $p$ and $q$, $|q| < Q$,
$$(g^q(x)-x-p) (q \rho(g)-p)>0.$$
\end{lemma}
We will now adapt the proof of the commensurability property \cite{Sw} with the objective of obtaining an expression for the commensurability constant. 
\begin{lemma} (Lemma 1.2. from \cite{Sw}) \label{1.2}
Suppose that $f: \field{R} \mapsto \field{R}$ is a lifting of a degree $1$ circle homeomorphism with an irrational rotation number $\rho(f)$. Assume that $f$ satisfies a cross-ratio inequality with respect to the cross-ratio modulus $(\ref{qmodulus})$ with bound $Q$. Let $p_n / q_n$ be a convergent of $\rho(f)$. Then there exists $K_1$,
\begin{equation}\label{eq:K1}
K_1=2 \left(1-\left(\mu^{-1}\left( {1 \over 8 \pi} Q^{-6} \ln {5 \over 3}   \right)   \right)^2    \right)^{-6},
\end{equation}
so that for every $x \in \field{R}$
$$ K_1^{-1} |f^{q_n}(x)-p_n-x| \le |f^{-q_n}(x)+p_n-x| \le K_1 |f^{q_n}(x)-p_n-x|.$$
\end{lemma}
\begin{proof}
Assume that $p_n / q_n < \rho(f)$. Otherwise we can consider $F(x)=-f(-x)+1$: the same estimate will follow from a similar argument for $F$ and $-x$ in place of $f$ and $x$, respectively.

Denote $p'/q'$ the next fraction larger than $p_n/q_n$ from the sequence of the best rational approximants with $q' \le q_n$. Such $p'/q'$ is larger than $\rho(f)$, otherwise
$$0<\rho(f)-{p' \over q'} < \rho(f)-{p_n \over q_n} \implies |q' \rho(f)-p'| < |q_n \rho(f)-p_n|$$
which contradicts the fact that $q_n$ is the closest return time for times smaller or equal to $q_n$.

Fix a point $x \in \field{R}$. From Lemma $\ref{f1.1}$, $f^{q_n}(x)-p_n>x$ and $f^{q'}(x)-p'<x$, therefore, $f^{-q'}(x)+p'>x$. Choose $n \ge 2$ such that $f^{-q'}(x)+p'    \in (f^{(n-1)q_n}(x)-(n-1) p_n, f^{n q_n}(x)-n p_n)$. 

Next, consider $F_s=f-s$ where $s$ is a non-negative number. As $s$ increases, points in the forward orbit of $x$ shift to the left, those in the backward orbit - to the right. Hence, there is a unique value $s^*$ of $s$ for which 
\begin{equation}\label{periodF}
F^{n q_n}_{s^*}(x)-n p_n=F^{-q'}_{s^*}(x)+p'.
\end{equation}
For simplicity, we will write $F$ in place of $F_{s^*}$. It follows from $(\ref{periodF})$, that such $F$ has the rotation number equal to
$$\rho'={n p_n + p' \over n q_n +q'}.$$

We have the following ordering of points (see \cite{Sw} for details)
\begin{eqnarray}
\label{ineq1} 0<F^{q_n}(x)-p_n<f^{q_n}(x)&-&p_n<F^{2 q_n}(x)-2 p_n,\\
\label{ineq2} 0>F^{-q_n}(x)+p_n>f^{-q_n}(x)&+&p_n>F^{-2 q_n}(x)+2 p_n.
\end{eqnarray}

Consider the exponential projection of the real line on the circle: $\pi(x)=e^{2 \pi i x}$. The periodic orbit $\pi(F^i(x))$ consists of $n q_n+q'$ points, while the points $\pi(x)$ and $\pi(F^{q_n}(x))$ are consecutive on the circle (again, see \cite{Sw} for details). 

Denote $\cI$ the collection of arcs on the circle with endpoints at two consecutive points of the orbit of $x$ by the projection of $F$. If $I_1, I_2 \in \cI$ are adjacent, and $F$ satisfies a cross-ratio inequality with respect to the cross-ratio modulus $(\ref{qmodulus})$ with bound $Q$, then 
\begin{equation}\label{Iratio}
{|I_1| \over |I_2|} \ge   \left(1-\left(\mu^{-1}\left( {1 \over 8 \pi} Q^{-6} \ln {5 \over 3}   \right)   \right)^2    \right)^2.
\end{equation}

This can be seen from the following argument. Choose $I_3 \in \cI$ adjacent to $I_2$ on the opposite side to  $I_1$. Lift $I_1, I_2, I_3$ to the line to get some intervals $(a,b), (b,c), (c,d)$ respectively. Observe that $|I_1| / |I_2| > {\bf Cr}(a,b,c,d)$. Choose the smallest $l$ such that $f^l(I_2)$ is the shortest arc in $\cI$. The configuration $(a,b,c,d),$ $\ldots,$ $(f^l(a),f^l(b),f^l(c),f^l(d))$ has the  intersection number at most $3$ (see \cite{Sw}). Also
$${\bf Cr} (f^l(a),f^l(b),f^l(c),f^l(d))=\!{|f^l(I_1)| \cdot |f^l(I_3)|  \over (|f^l(I_1)|\!+\!|f^l(I_2)|) \cdot (|f^l(I_3)|\!+\!|f^l(I_2)|)  }=\!
{1 \cdot 1  \over \left(1\!+\!{|f^l(I_2)|  \over |f^l(I_1)| } \right) \cdot  \left(1\!+\! {|f^l(I_2)|  \over |f^l(I_3)| } \right) } \!\ge {1 \over 4}.$$
Therefore, by  $(\ref{chi_const})$ and by Lemma $\ref{1.1}$,
\begin{equation}\label{rel1}
 Q^6 \chi (a,b,c,d) \ge \chi (f^l(a),f^l(b),f^l(c),f^l(d)) \ge { 1 \over \pi} \ln {5 \over 3}.
\end{equation}

Finally, recall the condition $(\ref{cond2})$:
$${|I_1| \over |I_2|} > {\bf Cr}(a,b,c,d) \implies 8 \mu\left(\sqrt{1- \sqrt{|I_1| \over |I_2|}  }   \right) > \chi(a,b,c,d).$$

Putting this together with $(\ref{rel1})$, we obtain 
\begin{equation}\label{eq:D}
 \mu\left(\sqrt{1- \sqrt{|I_1| \over |I_2|}  }   \right) \ge  {1 \over 8 \pi} Q^{-6} \ln {5 \over 3} \implies {|I_1| \over |I_2|} \ge \left(1-\left(\mu^{-1}\left( {1 \over 8 \pi} Q^{-6} \ln {5 \over 3}   \right)   \right)^2    \right)^2 \equiv D,
\end{equation}
which is just $(\ref{Iratio})$. It follows immediately, that the four intervals with the limiting points
$$F^{-2 q_n}(x)+2 p_n, F^{-q_n}(x)+p_n, x, F^{q_n}(x)-p_n, F^{2 q_n}(x)-2 p_n$$
have lengths comparable with the factor $ D^{-3}$. Together with the inequalities $(\ref{ineq1})$ and $(\ref{ineq2})$, the Lemma follows with 
$$K_1=2 D^{-3}.$$

\end{proof} \qed

We will now outline the proof of the quasisymmetric property of the conjugacy. 

\begin{proposition} (Prop. 1 in \cite{Sw}) \label{prop1}
Let $f: \field{R} \mapsto \field{R}$ be a lifting of degree $1$ circle homeomorphism with an irrational rotation number $\rho(f)$. Suppose that  $f$ satisfies a cross-ratio inequality with bound $Q$.

Then, there exist a lift of degree $1$ circle homeomorphism $h: \fR \mapsto \fR$, which conjugates $f$ to a translation by $\rho(f)$:
$$f(h(x))=h(x+\rho(f)).$$

Furthermore, if $\rho(f)$ is of bounded type, $a_n \le B$ for all $n$, then $h$ is quasisymmetric with the quasisymmetric constant
\begin{equation}\label{Kconst}
K_2=\max\left\{2, K_1^{B+1} \right\},
\end{equation}
where $K_1$ is as in $(\ref{eq:K1})$.

\end{proposition}
\begin{proof}
By assumption $\rho(f)$ is of bounded type: there exists a positive integer $B$ such that  $\rho(f)=[a_1,a_2,\ldots]$, $\sup_{i\ge 1}a_i \le B$. Recall that
$$a_n|q_n \rho(f)-p_n| \le |q_{n-1} \rho(f) -p_{n-1}| \le (a_n+1) | q_n \rho(f) - p_n|$$
for all $n \ge 1$ with the convention $p_0 / q_0 =1/0$.  Together with Lemma $\ref{f1.1}$, this leads to
\begin{equation}\label{rel8} 
|f^{\epsilon (a_n+1) q_n}(x)-\epsilon (a_n+1) p_n -x| > |f^{-\epsilon q_{n-1}}(x)+\epsilon p_{n-1}-x|
\end{equation}
for all $n \ge 1$, $x \in \fR$ and $\epsilon \in \{1,-1\}$,

Take $1 \ge t>0$, and choose $n \ge 0$ so that 
$$|q_{n+1} \rho(f)-p_{n+1}| < t \le |q_{n} \rho(f)-p_{n}|,$$
then
\begin{equation}\label{containments}
h(x+t) \!\in\! \left(f^{-\epsilon q_{n+1}}(h(x))\!+\!\epsilon p_{n+1}, f^{\epsilon q_n}(h(x)) \! - \! \epsilon p_n \right], \  h(x-t) \! \in \! \left[f^{-\epsilon q_{n}}(h(x))\! +\! \epsilon p_{n}, f^{\epsilon q_{n+1}}(h(x)) \! -\! \epsilon p_{n+1} \right)
\end{equation}
where $\epsilon=(-1)^n$. We can now use relation $(\ref{rel8})$ in the following computation
\begin{eqnarray}
\nonumber { |f^{-\epsilon q_{n+1}}(h(x))+\epsilon p_{n+1}-h(x)|     \over  |f^{-\epsilon q_{n}}(h(x))+ \epsilon p_{n}-h(x)|    } \le &{| h(x+t)-h(x)| \over |h(x)-h(x-t)|}& \le { |f^{\epsilon q_{n}}(h(x))-\epsilon p_{n}-h(x)|     \over  |f^{\epsilon q_{n+1}}(h(x))-\epsilon p_{n+1}-h(x)|    } \\
\nonumber { |f^{-\epsilon q_{n+1}}(h(x))+\epsilon p_{n+1}-h(x)|     \over  |f^{\epsilon (a_{n+1}+1) q_{n+1}}(h(x))-(a_{n+1}+1) \epsilon p_{n+1}-h(x)|    } \le &{| h(x+t)-h(x)| \over |h(x)-h(x-t)|}& \le { |f^{-\epsilon (a_{n+1}+1) q_{n+1}}(h(x))+(a_{n+1}+1)\epsilon p_{n+1}-h(x)|     \over  |f^{\epsilon q_{n+1}}(h(x))-\epsilon p_{n+1}-h(x)|    }. 
\end{eqnarray}
By Lemma $\ref{1.2}$ any two adjacent intervals with the end points $f^{k q_{n+1}}(h(x))-k p_{n+1}$ and $f^{(k+1) q_{n+1}}(h(x))-(k+1) p_{n+1}$ for $k \in \fZ$ are comparable with the constant $K_1$. Therefore, 
$$K_1^{-B-1} \le {|h(x+t)-h(x) | \over |h(x)-h(x-t)|} \le K_1^{B+1}$$
for $0 \le t \le 1$.

Suppose $t>1$, and let $m$ be its integer part, then since $h(x+m)=h(x)+m$, we have
$${m \over m+1} \le {|h(x+t)-h(x) | \over |h(x)-h(x-t)|} \le {m+1 \over m}.$$

Therefore, $h$ is $\max\left\{2, K_1^{B+1} \right\}$-quasisymmetric.

\end{proof} \qed

We can now finish the proof of the Main Theorem. It follows from the Blaschke model for the Siegel disk, that 
$$\qp=f \circ m \circ h^{-1} \circ \rot \circ h \circ  m^{-1} \circ f^{-1}$$ 
on $\partial \sd$.  Recall, that according to the theorem on page 9 entire quasiconformal maps are quasisymmetric. We, therefore, have that the map $g=f \circ m \circ h^{-1}$, $m$ being the M\"obius transformation $(\ref{mbt})$, satisfies
$${ |g(x+\delta)-g(x)| \over |g(x)-g(x-\delta)| } \le \eta\left( {  |m(h^{-1}(x+\delta))-m(h^{-1}(x))| \over |m(h^{-1}(x))-m(h^{-1}(x-\delta))| }  \right) \le \lambda(M)^{2 M} (C(\delta) K_2)^M,$$
where $\lambda(M)$ has been defined in $(\ref{def:circular_distortion})$, $M=2K_2-1$, $K_2$ is as in $(\ref{Kconst})$ and $C(\delta) \rightarrow 1$ as $\delta\rightarrow 0$.

Now, let $\theta=[a_1,a_2, \ldots]$ be a quadratic irrational, such that for all $i$ larger than some $N \ge 1$, we have $a_i=a_{i+s}$.  Suppose that $\alpha$, as in $(\ref{quad_ir})$, is larger than $1/2$ and the period is even.  Then, $[\qp^{q_{n+s}}(1),1] \subset [1,\qp^{q_{n}}(1)]$ and $[\rot^{q_{n+s}}(1),1] \subset [1,\rot^{q_{n}}(1)]$, while the restriction of the conjugacy $g$ to $[1,\rot^{q_{n}}(1)]$ is  $\lambda(M)^{2 M} (\tilde{C}(n) K_2)^M$-quasisymmetric with $\lim_{n \rightarrow \infty} \tilde{C}(n)=1$. Now, the bound $(\ref{b2.1})$  follows immediately from $3)$ of Lemma $\ref{SRbounds}$.

Next, suppose the period is odd, then ${[\qp^{q_{n+s}}(1),1] \nsubseteq [1,\qp^{q_{n}}(1)]}$, rather,  one can consider the intervals
$${[1,\qp^{q_{n+2s}}(1)] \subset [1,\qp^{q_{n}}(1)]}.$$ 
We recall, that $|\rot^{q_{n+2s}}(1)-1|=\alpha^2 |\rot^{q_{n}}(1)-1|$, then, if $\alpha>1/\sqrt{2}$, we obtain from  $3)$ of Lemma $\ref{SRbounds}$ that
$${|\qp^{q_{n+2 s}}(1)-1| \over |\qp^{q_{n}}(1)-1|} \le {\hat{C}(n) \over K^{-1} \left(1 \over 1+K  \right)^{\log_2{\alpha^2 \over 1-\alpha^2}+1}+1}, \quad \lim_{n \rightarrow \infty} \hat{C}(n)=1$$
$K$ being as in the Main Theorem, and since, as $n \rightarrow \infty$, 
\begin{equation}\label{square}
{{|\qp^{q_{n+2 s}}(1)-1| \over |\qp^{q_{n}}(1)-1|}  \over \left({|\qp^{q_{n+s}}(1)-1| \over |\qp^{q_{n}}(1)-1|}\right)^2} \rightarrow 1,
\end{equation}
the bound $(\ref{b1.1})$ follows. If the period $s$ is even, and $\alpha \le 1/2$, then the bound $(\ref{b2.2})$ follows from $1)$ of Lemma $\ref{SRbounds}$. The bound $(\ref{b1.2})$ follows the same result $1)$ of Lemma $\ref{SRbounds}$, and $(\ref{square})$ above.

The value of $Q=8$ (see Section $\ref{cr_ineq}$) has been used in the expression $(\ref{eq:D})$, and leads to the factor $8^{-7}$ in the argument of the inverse modulus of the Gr\"otzsch's extremal domain in  $(\ref{eq:Ks})$.

Notice, that we are not using $2)$ of Lemma $\ref{SRbounds}$ in case of odd $s$. In this case the intervals  $[\qp^{q_{n+s}}(1),1]$ and  $[1,\qp^{q_{n}}(1)]$ share a single boundary point, and it would seem natural to apply $2)$ of Lemma $\ref{SRbounds}$. However, an expression $1 / \left( (1+K^{-1})^{\left[s \log_2 \nu \right]}-1  \right)$ can not be bounded from above by an expression of the form $\delta_2^s$ for some $\delta_2<1$, indeed,
$${1  \over   (1+K^{-1})^{\left[s \log_2 \nu \right]}-1} \le {1  \over   1+\left[s \log_2 \nu \right] K^{-1}-1}   \le {K \over \left[s \log_2 \nu \right]}.$$
We, nonetheless, chose to keep case $2)$ in Lemma $\ref{SRbounds}$ for completeness.

\section{Cross-ratio inequality} \label{cr_ineq}

In this Section we will recall  \'Swi\c{a}tek's proof of the fact that the cross-ratio modulus $(\ref{qmodulus})$ satisfies the cross-ratio inequality with a definite bound. Our goal will be to adapt the proof to the particular case of the Blaschke product, and to show that in this case the bound can be chosen to be ``absolute'' - independent of the rotation number, as long as it is a quadratic irrational.

Let $g(z)= {1 \over 2 \pi i} \ln(\tilde{Q}_\theta(\exp(2 \pi i z)))$ be the exponential lift of the modified Blaschke product.

\begin{lemma}
The exponential lift of the modified Blaschke product $\tilde{Q}_\theta$ satisfies the following cross-ratio inequality with respect to the cross-ratio modulus $(\ref{qmodulus})$:
$$\prod_{i=i}^n { \chi(g(a_i  ), g(b_i  ), g(c_i  ), g(d_i ))  \over  \chi(a_i,b_i,c_i,d_i)} \le 8.$$
\end{lemma}
\begin{proof}
  The lift $g$ is holomorphic in the upper half-plane, and conformal in all quadrilaterals that do not contain the critical points $\ldots,-3,-2,-1,0,1,2,3,\ldots$. 

Suppose that $(a_i,b_i,c_i,d_i)$ is some allowable configuration of quadruples of points (in particular, at most one interval $(a_{i^*} , d_{i^*} )$ contains a critical point). We first consider the contribution to the cross-ratio inequality of all non-critical intervals. The map $g$ is conformal on the corresponding quadrilaterals, and we get immediately
\begin{equation}\label{contrib1}
 \prod_{{\rm non-critical \ intervals}} { \chi(g(a_i  ), g(b_i  ), g(c_i  ), g(d_i ))  \over  \chi(a_i,b_i,c_i,d_i)}=1.
\end{equation}

Now, consider points $(a_{i^*} , b_{i^*} , c_{i^*} , d_{i^*} )$ in the critical interval,  choose $k >2 |g(a_{i^*}) -g(d_{i^*})|$ and a quadrilateral $Q' \subset \{z \in \field{C}: |\Im z| < k \} \cap Q($$g(a_{i^*}  ),$ $g(b_{i^*}  ),$ $g(c_{i^*}  ), g(d_{i^*} ))$. Its ring domain is contained in $\{z \in \field{C}: |\Im z| < k \} $, and
\begin{equation}\label{quarter} \mod  R(Q')  \ge  {1- {|g(a_{i^*})-g(d_{i^*})| \over k}  \over  2  \mod Q(g(a_{i^*}  ), g(b_{i^*}  ), g(c_{i^*}  ), g(d_{i^*} ))  } \ge  {1 \over 4} \chi(g(a_{i^*}  ), g(b_{i^*}  ), g(c_{i^*}  ), g(d_{i^*} )).
 \end{equation}

Consider $R'=g^{-1}(R(Q'))$. If the critical point $c$ lies in $(b_{i^*}, c_{i^*}  )$, then 
$$ \mod  R'={1 \over d} \mod  R(Q'),$$  
 where $d$ is the degree of branching of $g$ at the critical point. For the modified Blaschke product $d=2$. If the critical point $c$ lies in $(a_{i^*}, b_{i^*}  )$ or $(c_{i^*}, d_{i^*} )$, then one can find a curve $\gamma \ni g(c)$ that divides the ring domain $R(Q')$ in two ring domains $R_{outer}$ and $R_{inner}$. If $\mod   R_{outer} > \mod   R_{inner}$, set $R'=g^{-1}(R_{outer})$, then  
$$ \mod  R' \ge {1 \over 2 d} \mod  R(Q').$$  
If $\mod  R_{inner} > \mod  R_{outer}$, set $R'=g^{-1}(R_{inner})$, then 
$$ \mod  R' \ge {1 \over 2} \mod  R(Q').$$  
Therefore, in either case,
\begin{equation}
 2 d \mod{R'} \ge  \mod{R(Q')}.
 \end{equation}
Finally, 
$$ \chi(a_{i^*},b_{i^*},c_{i^*},d_{i^*})=2 \mod{R(Q(a_{i^*},b_{i^*},c_{i^*},d_{i^*}))} \ge 2 \mod R',$$
which, we use, together with  $\eqref{quarter}$, in the estimate for the critical interval:
\begin{equation}\label{contrib2}
{ \chi(g(a_{i^*}  ), g(b_{i^*}  ), g(c_{i^*} ), g(d_{i^*} ))  \over  \chi(a_{i^*} , b_{i^*} , c_{i^*} , d_{i^*})} \le { 4 \mod R(Q') \over  2 \mod R'}  \le 4 d. 
\end{equation}

We combine $(\ref{contrib1})$ with $(\ref{contrib2})$ to obtain the bound in the claim of the Lemma.
\end{proof}
\qed

\end{document}